\documentclass{amsart}
\usepackage{amssymb}
\usepackage{pictexwd}

\newtheorem{theorem}{Theorem}[section]
\newtheorem{proposition}[theorem]{Proposition}
\newtheorem{corollary}[theorem]{Corollary}
\newtheorem{lemma}[theorem]{Lemma}

\newtheorem{preremark}[theorem]{Remark}
\newtheorem{predefinition}[theorem]{Definition}
\newtheorem{preexample}[theorem]{Example}
\newenvironment{remark}{\begin{preremark}\rm}{\end{preremark}}
\newenvironment{definition}{\begin{predefinition}\rm}{\end{predefinition}}
\newenvironment{example}{\begin{preexample}\rm}{\end{preexample}}

\newcounter{itemcounter}
\newenvironment{items}{
   \begin{list}{\alph{itemcounter})}
   {\usecounter{itemcounter}\setlength{\topsep}{3 pt}
   \setlength{\leftmargin}{2 em}
   \setlength{\partopsep}{0 pt}\setlength{\itemsep}{0 pt}
      \setlength{\labelwidth}{2 em}
   }}{\end{list}}

\def\LT{\mathop{\rm LT}\nolimits}

\def\DF{\mathop{\rm DF}\nolimits}
\def\LF{\mathop{\rm LF}\nolimits}
\def\BF{\mathop{\rm BF}\nolimits}
\def\BT{\mathop{\rm BT}\nolimits}
\def\ND{\mathop{\rm ND}\nolimits}
\def\AS{\mathop{\rm AS}\nolimits}
\def\Mat{\mathop{\rm Mat}\nolimits}

\def\Supp{\mathop{\rm Supp}\nolimits}
\def\indO{{\mathop{\rm ind}\nolimits}_{\mathcal{O}}}
\def\maxdeg{\mathop{\rm maxdeg}\nolimits}

\let\epsilon=\varepsilon
\let\rho=\varrho
\let\To=\longrightarrow
\def\TTo#1{\mathop{\longrightarrow}\limits ^{#1}}
\def\longiso{\,\smash{\TTo{\lower 7pt\hbox{$\scriptstyle\sim$}}}\,}
\def\hom{^{\rm hom}}
\def\tfrac #1#2{{\textstyle\frac{#1}{#2}}}

\def\cocoa{\mbox{\rm
   C\kern-.13em o\kern-.07 em C\kern-.13em o\kern-.15em A}}

\begin{document}

\title{Deformations of Border Bases}

\author{Martin Kreuzer}
\address{Fakult\"at f\"ur Informatik und Mathematik, Universit\"at Passau,
D-94030 Passau, Germany}\email{kreuzer@uni-\,passau.de}

\author{Lorenzo Robbiano}
\address{Dipartimento di Matematica, Universit\`a di Genova, Via Dodecaneso 35,
I-16146 Genova, Italy}\email{robbiano@dima.unige.it}

\date{\today}

\begin{abstract}
Border bases have recently attracted a lot of attention.
Here we study the problem of generalizing one of the main
tools of Gr\"obner basis theory, namely the flat deformation to
the leading term ideal, to the border basis setting.
After showing that the straightforward approach based on the
deformation to the degree form ideal works only under
additional hypotheses, we introduce border basis schemes
and universal border basis families. With their help the
problem can be rephrased as the search for a certain rational
curve on a border basis scheme. We construct the system of generators
of the vanishing ideal of the border basis scheme in different ways
and study the question of how to minimalize it.
For homogeneous ideals, we also introduce a homogeneous border basis
scheme and prove that it is an affine space in certain cases.
In these cases it is then easy to write down the desired
deformations explicitly.
\end{abstract}


\maketitle


\section{Introduction}

Let $I$ be a zero-dimensional ideal in a polynomial ring
$P=K[x_1,\dots,x_n]$ over a field~$K$, and let~$\mathcal{O}
=\{t_1,\dots,t_\mu\}$ be an order ideal, i.e.\ a finite set of
power products in~$P$ which is closed under taking divisors.
An {\it $\mathcal{O}$-border basis} of~$I$ is a set of polynomials
$G=\{g_1,\dots,g_\nu\}$ of the form $g_j=b_j -\sum_{i=1}^\mu c_{ji}t_i$,
where $\{b_1,\dots,b_\nu\}$ is the border $\partial\mathcal{O}=
(x_1\mathcal{O}\cup \cdots\cup x_n\mathcal{O})\setminus \mathcal{O}$
of~$\mathcal{O}$ and $c_{ji}\in K$, such that~$I$ is generated by~$G$
and $\mathcal{O}$ is a $K$-vector space basis of~$P/I$. In recent years
border bases have received considerable attention (see for instance~\cite{KK1},
\cite{KK2}, \cite{KKR}, \cite{M}, and~\cite{S}).
This is due to several reasons.

\begin{items}
\item[(1)] Border bases generalize Gr\"obner bases: if one takes for
$\mathcal{O}$ the complement of a leading term ideal of~$I$
with respect to some term ordering~$\sigma$, the corresponding border basis
contains the reduced $\sigma$-Gr\"obner basis of~$I$.

\item[(2)] Border bases are more suitable for dealing with computations
arising from real world problems. They are more stable with respect
to small variations in the coefficients of the polynomials generating~$I$
and permit symbolic computations with polynomial systems having
approximate coefficients (see for instance~\cite{AFT}, \cite{HKPP}, and~\cite{S}).

\item[(3)] Border bases are in general much more numerous than reduced Gr\"obner
bases. For instance, if the given ideal~$I$ is invariant under the action of a group of
symmetries, it is sometimes possible to find a border basis having these
symmetries, but not a Gr\"obner basis.
\end{items}

The starting point for this paper is our attempt to
generalize one of the fundamental results of Gr\"obner
basis theory to the border basis setting, namely the
fact that there exists a flat deformation from~$I$ to its
leading term ideal~$\LT_\sigma(I)$. More precisely, we are
looking at the following result. (Here and in the following
we use the notation introduced in~\cite{KR1} and~\cite{KR2}.)

Given a term ordering~$\sigma$, the ring~$P$ can be graded by a
row of positive integers $W=(w_1\;\cdots\;w_n)$, i.e.\ by
letting $\deg_W(x_i)=w_i$, such that the leading term ideal
$\LT_\sigma(I)$ equals the degree form ideal $\DF_W(I)$.
Using a homogenizing indeterminate $x_0$ and the grading of
$\overline{P}=K[x_0,\dots,x_n]$ given by $\overline{W}=
(1\;w_1\;\cdots\;w_n)$, the canonical $K$-algebra homomorphism
$\Phi: K[x_0] \;\To\; \overline{P}/I\hom$ satisfies

\begin{items}
\item[(1)] The ring $\overline{P}/I\hom$ is a free
$K[x_0]$-module.

\item[(2)] There are isomorphisms of $K$-algebras
$\overline{P}/(I\hom +(x_0)) \cong
P/\DF_W(I)$ and $\overline{P}/(I\hom +
(x_0-c)) \cong P/I$ for every $c\in K\setminus \{0\}$.
\end{items}

We express this by saying that there is a flat deformation
from~$I$ to~$\DF_W(I)$, and thus to~$\LT_\sigma(I)$.
In geometric jargon, we can say that, in the Hilbert scheme
parametrizing affine schemes of length $\dim_K(P/I)$,
the affine scheme defined by~$I$ is connected to the scheme defined
by~$\DF_W(I)$ via a rational curve parametrized by~$x_0$.
Thus the starting point for this paper is the question
whether there exists a flat deformation from a zero-dimensional
ideal~$I$ given by an $\mathcal{O}$-border basis
$G=\{g_1,\dots,g_\nu\}$ as above to its border term ideal
$\BT_\mathcal{O}=(b_1,\dots,b_\nu)$.

The direct approach taken in
Section~\ref{Deformation to the Border Form Ideal}
is to try to imitate Gr\"obner basis theory and to
use the flat deformation to the degree form ideal
we just recalled. Unfortunately, this approach does not
succeed in all cases, but only under the additional
assumption that~$\mathcal{O}$ has a $\maxdeg_W$ border,
i.e.\ that no term in~$\mathcal{O}$ has a larger degree than
a term in the border~$\partial\mathcal{O}$.

Therefore it is necessary to dig deeper into the problem
and find other ways of constructing the desired flat
deformations. In Section~\ref{The Border Basis Scheme}
we take a step back and view the task from a more global
perspective. All zero-dimensional ideals having an
$\mathcal{O}$-border basis can be parametrized by a
scheme~$\mathbb{B}_{\mathcal{O}}$ which we call the
{\it $\mathcal{O}$-border basis scheme}.
Using the condition that the generic multiplication matrices
have to commute, we give explicit equations
defining~$\mathbb{B}_{\mathcal{O}}$ in a suitable affine space
(see Definition~\ref{defBBS}).

A moduli space such as the border basis scheme usually comes
together with a universal family: this is a morphism
from~$\mathbb{B}_{\mathcal{O}}$ to another scheme whose fibers are
precisely the schemes defined by the ideals having an $\mathcal{O}$-border
basis. The fundamental result about this {\it universal border basis
family} is that it is flat. In fact, in Theorem~\ref{universal}
we give an elementary, explicit proof that
$\mathcal{O}$ is a basis for the entire family, viewed as a module over
the coordinate ring of the border basis scheme.
Hence the construction of the desired flat deformation of an
ideal to its border term ideal is equivalent to finding
suitable rational curves on the border basis scheme
(see Corollay~\ref{ratcurve}).

To examine the border basis scheme further, we have a more detailed
look at the system of generators of its vanishing ideal in
Section~\ref{Defining Equations for the Border Basis Scheme}.
The technique of lifting neighbor syzygies (introduced in~\cite{KK1}
and~\cite{S}, and independently in~\cite{Hu2}) provides us
with a different way of constructing  a system of generators
of~$I(\mathbb{B}_{\mathcal{O}})$ (see Proposition~\ref{altgenBBS}).
Using suitable examples, including the well-known Example~\ref{exa1xyz}
of a singularity on a Hilbert scheme, we disprove several
claims in~\cite{S} with respect to the possibility of
removing redundant generators from this system.
On~the positive side, in Proposition~\ref{removing} we provide a
criterion for eliminating some unnecessary generators.

The final Section~\ref{The Homogeneous Border Basis Scheme}
introduces the homogeneous border basis scheme
$\mathbb{B}_{\mathcal{O}}\hom$. It parametrizes
all homogeneous zero-dimensional ideals having an $\mathcal{O}$-border
basis and is obtained from the border basis scheme by intersecting it
with a suitable linear space.
Our main result about $\mathbb{B}_{\mathcal{O}}\hom$ is that it is
an affine space (and not only isomorphic to an affine space)
if~$\mathcal{O}$ has a $\maxdeg_W$ border (see Theorem~\ref{homcommute}).
This theorem is a nice tool which can be employed
to produce good deformations (see Example~\ref{exdefcontinued}) and
to recreate the construction of reducible Hilbert schemes
(see Example~\ref{exIarrobino}).

Here we close this introduction by pointing out that all computations were
done using the computer algebra system~\cocoa (see~\cite{CoCoA}) and that
even great artists can be too pessimistic at times.

\begin{flushright}
\small\it
Deformations simply do not exist.\\
\rm (Pablo Picasso)
\end{flushright}

\bigskip

%
%

\section{Deformation to the Border Form Ideal}
\label{Deformation to the Border Form Ideal}

One of the fundamental results of Gr\"obner basis theory is that there
exists a flat deformation of a polynomial ideal to its leading term ideal.
This deformation is achieved by taking a Gr\"obner basis of the ideal,
viewing it as a Macaulay basis with respect to a suitably chosen $\mathbb
N$-grading, homogenizing it, and letting the homogenizing indeterminate tend
to zero. An analogous fact for border bases of
zero-dimensional polynomial ideals is not known in general. In this section
we shall prove some partial results in this direction.

In the following we let $K$ be a field, $P=K[x_1,\dots,x_n]$ a polynomial
ring, and $I\subset P$ a zero-dimensional ideal. Recall that an {\em order
ideal}~$\mathcal{O}$ is a finite set of terms in $\mathbb T^n=\{
x_1^{\alpha_1}\cdots x_n^{\alpha_n} \mid \alpha_i\ge 0\}$ such that all
divisors of a term in~$\mathcal{O}$ are also contained in~$\mathcal{O}$. The set
$\partial\mathcal{O}=(x_1\mathcal{O} \cup\cdots\cup x_n\mathcal{O}) \setminus \mathcal{O}$
is called the {\em border} of~$\mathcal{O}$. By repeating this construction, we define
the {\it higher borders} $\partial^i\mathcal{O}$ for $i\ge 1$ and we let
$\partial^0\mathcal{O}=\mathcal{O}$. The number $\indO(t)=\min\{i\ge 0
\mid t\in \partial^i \mathcal{O}\}$ is called the {\em $\mathcal{O}$-index} of
a term $t\in\mathbb{T}^n$.

\begin{definition}
Let $\mathcal{O}=\{t_1,\dots,t_\mu\}$ be an order ideal and $\partial\mathcal{O}
=\{b_1,\dots,b_\nu\}$ its border.

\begin{items}
\item A set of polynomials $G=\{g_1,\dots,g_\nu\}\subseteq I$
is called an {\em $\mathcal{O}$-border prebasis} of~$I$ if it is of the form
$g_j=b_j-\sum_{i=1}^\mu a_{ij}t_i$ with $a_{ij}\in K$.

\item An $\mathcal{O}$-border prebasis of~$I$ is called an {\em $\mathcal{O}$-border
basis} of~$I$ if $P=I\oplus \langle \mathcal{O}\rangle_K$.

\item For a polynomial $f=c_1 u_1+\cdots+c_s u_s\ne 0$ with $c_i\in K\setminus \{0\}$
and $u_i\in\mathbb T^n$, the polynomial $\BF_{\mathcal{O}}(f)=\sum_{\{i\mid \indO(u_i)\,
\hbox{\scriptsize max.}\}}c_i u_i$ is called the {\em border form} of~$f$.

\item The ideal $\BF_{\mathcal{O}}(I)= ( \BF_{\mathcal{O}}(f)\mid f \in
I\setminus \{0\} )$ is called the {\em border form ideal} of~$I$.

\item The monomial ideal generated by~$\partial\mathcal{O}$ is called the {\em border
term ideal}\/ of~$\mathcal{O}$ and is denoted by~$\BT_{\mathcal{O}}$.

\end{items}
\end{definition}

Notice that
if~$I$ has an $\mathcal{O}$-border basis, its border from ideal is
$\BF_{\mathcal{O}}(I)=\BT_{\mathcal{O}}$.
Thus our goal is to use a
border basis of~$I$ to deform the ideal to its border form ideal. If the
order ideal is of the form $\mathcal{O}_\sigma(I)=\mathbb T^n\setminus
\LT_\sigma(I)$ for some term ordering~$\sigma$, the Gr\"obner deformation
can be used as follows.

\begin{proposition}\label{TOdeform}
Let $\sigma$ be a term ordering, let $G=\{g_1,\dots,g_\nu\}$ be the
$\mathcal{O}_\sigma(I)$-border basis of~$I$, and let $b_i$ the border term
in the support of~$g_i$ for $i=1,\dots,\nu$.

\begin{items}
\item There exist weights $W=(w_1,\dots,w_n)\in (\mathbb N_+)^n$
such that $b_j=\DF_W(g_j)$ and~$G$ is a Macaulay basis of~$I$ with respect
to the grading given by~$W$.

\item Let $\overline{P}=K[x_0,\dots,x_n]$ be graded by
$\overline{W}=(1,w_1,\dots,w_n)$. Then the ring $\overline{P}/I\hom=
\overline{P}/ ( g_1\hom,\dots, g_\nu\hom )$ is a graded free
$K[x_0]$-module.

\end{items}

\noindent In particular, we have a {\em flat family} $K[x_0] \To \overline{P}/I\hom$
whose {\em general fiber} is isomorphic to $P/I\cong
\overline{P}/(I\hom+( x_0-1))$,
where $I= ( g_1,\dots,g_\nu)$,
and whose {\em special fiber} is isomorphic to~$P/\BT_{\mathcal{O}_\sigma(I)} \cong
\overline{P}/(I\hom+( x_0))$.
\end{proposition}

\begin{proof}
The first claim in~a) follows from~\cite{E}, Prop.\ 15.16. The second claim
in~a) is then a consequence of~\cite{KR2}, Props.\ 6.4.18 and~4.2.15.
The remaining claims follow from~a) and~\cite{KR2}, Thm.\ 4.3.22 and Prop.\ 4.3.23.
\end{proof}

For more general order ideals~$\mathcal{O}$, i.e.\ for order ideals
which are not necessarily of the form $\mathcal{O}_\sigma(I)$, one strategy is
to deform a given $\mathcal{O}$-border basis of~$I$ first to a border basis
of the degree form ideal~$\DF_W(I)$ of~$I$ with respect to a suitably chosen grading.
A border basis of~$\DF_W(I)$ is always homogeneous, as the following lemma shows.

\begin{lemma}
Let $P$ be graded by a matrix $W\in\Mat_{m,n}(\mathbb{Z})$, let~$\mathcal{O}$
be an order ideal, and let $I\subset P$ be a homogeneous ideal which has
an $\mathcal{O}$-border basis. Then this $\mathcal{O}$-border basis of~$I$
consists of homogeneous polynomials.
\end{lemma}

\begin{proof}
Let $\mathcal{O}=\{t_1,\dots,t_\mu\}$, let $b_j\in \partial\mathcal{O}$,
and let $g_j=b_j -\sum_{i=1}^\mu c_{ij}t_i$ be the corresponding border basis
element, where $c_{ij}\in K$. If we restrict the sum to those indices~$i$ for
which $\deg_W(t_i)=\deg_W(b_j)$, we obtain a homogeneous element of~$I$ of the form
$\tilde g_i=b_j -\sum_k c_{ik}t_k$. Now the uniqueness of the $\mathcal{O}$-border
basis of~$I$ (cf.~\cite{KR2}, 6.4.17) implies $g_i=\tilde g_i$.
\end{proof}

As for our idea to deform a border basis of~$I$ to a homogeneous border basis
of~$\DF_W(I)$, we have the following result.

\begin{theorem}{\bf (Deformation to the Degree Form Ideal)}\label{DFdeform}\\
Let $W=(w_1,\dots,w_n)\in\Mat_{1,n}(\mathbb N_+)$ be a row of positive integers,
let~$P$ be graded by~$W$, and let $I\subset P$ be a zero-dimensional ideal.
Then the following conditions are equivalent.

\begin{items}
\item The ideal~$I$ has an $\mathcal{O}$-border basis, say
$G=\{g_1,\dots,g_\nu\}$, and we have $b_j \in \Supp(\DF_W(g_j))$
for $j=1,\dots,\nu$.

\item The degree form ideal $\DF_W(I)$ has an $\mathcal{O}$-border
basis.
\end{items}

If these conditions are satisfied, the $\mathcal{O}$-border basis of $\DF_W(I)$ is
$\DF_W(G)=\{\DF_W(g_1),\dots,\DF_W(g_s)\}$ and
there is a flat family $K[x_0] \To \overline{P}/I\hom$ whose
general fiber is isomorphic to~$P/I$, where $I= (g_1,\dots,g_\nu)$,
and whose special fiber is isomorphic to $P/\DF_W(I)$, where $\DF_W(I)=
(\DF_W(g_1),\dots,\DF_W(g_\nu))$.
\end{theorem}

\begin{proof}
First we show that a) implies b). Since $G$ is an $\mathcal{O}$-border basis of~$I$
and since $b_j \in \Supp(\LF_W(g_j))$ for $j=1,\dots,\nu$, the set $\DF_W(G)=
\{\DF_W(g_1),\dots,\DF_W(g_\nu)\}$ is an $\mathcal{O}$-border prebasis
of the ideal $J=(\DF_W(g_1),\dots,\DF_W(g_\nu))$.
By the Border Division Algorithm (see \cite{KR2}, Prop.~6.4.11),
the residue classes of the elements of~$\mathcal{O}$
generate the $K$-vector space $P/J$. Together with $J\subseteq \DF_W(I)$,
this shows
$$
\#\mathcal{O}=\dim_K(P/I)= \dim_K(P/\DF_W(I)) \le \dim_K(P/J) \le \#\mathcal{O}.
$$
Therefore we get $J=\DF_W(I)$ and the residue classes of the elements of~$\mathcal{O}$
are a $K$-basis of~$P/\DF_W(I)$. From this the claim follows immediately.

Now we prove that b) implies a).
Let $\sigma$ be a term ordering on~$\mathbb T^n$ which is compatible
with the grading defined by~$W$, and let $H=\{h_1,\dots,h_\nu\}$
be the $\mathcal{O}_\sigma(I)$-border basis of~$I$.
For the purposes of this proof, we may consider~$\mathcal{O}$
and~$\mathcal{O}_\sigma(I)$ as deg-ordered tuples (see~\cite{KR2}, 4.5.4).

The fact the~$H$ is a $\sigma$-Gr\"obner basis of~$I$
implies by~\cite{KR2}, 4.2.15 that~$H$ is a Macaulay basis of~$I$
with respect to the grading given by~$W$. Then~\cite{KR2}, 4.3.19
shows that $I\hom$ is generated by~$\{h_1\hom,\dots, h_\nu\hom\}$,
and by~\cite{KR2}, 4.3.22 the ring $\overline{P}/I\hom$
is a graded free $K[x_0]$-module, where $K[x_0]$ is graded by $\deg(x_0)=1$
and $\overline{P}=K[x_0,\dots,x_n]$ is graded by
$\overline{W}=(1,w_1,\dots,w_n)$. More precisely, the proof of~\cite{KR2},
4.3.22 shows that the residue classes $\overline{\mathcal{O}_\sigma(I)}$
form a homogeneous $K[x_0]$-basis of this graded free module.
Since the residue classes $\overline{\mathcal{O}}$ are homogeneous elements
in $\overline{P}/I\hom$, we can write
$\overline{\mathcal{O}} = \mathcal{A} \cdot \overline{\mathcal{O}_\sigma(I)}$
with a homogeneous matrix $\mathcal{A}\in \Mat_{\nu}(K[x_0])$
(see~\cite{KR2}, 4.7.1 and~4.7.3).

By the hypothesis, $\DF_W(I)$ has an $\mathcal{O}$-border basis. Thus the
residue classes of the elements of~$\mathcal{O}$ are a homogeneous $K$-basis of
$P/\DF_W(I)$. Since also the residue classes of the elements of $\mathcal{O}_\sigma(I)$
are a homogeneous $K$-basis of this ring, the degree tuples
of~$\mathcal{O}$ and of~$\mathcal{O}_\sigma(I)$ are identical.
Therefore the matrix~$\mathcal{A}$ is a block matrix of the form
$$
\mathcal{A}=\begin{pmatrix}
\mathcal{A}_{11} & \mathcal{A}_{12} & \cdots & \mathcal{A}_{1q}\\
0 & \ddots & \ddots & \vdots \\
\vdots & \ddots & \ddots & \mathcal{A}_{q{-}1\,q} \\
0 & \cdots & 0 & \mathcal{A}_{qq}
\end{pmatrix}
$$
with square matrices $\mathcal{A}_{ii}$ having constant entries.
Hence we have $\det(\mathcal{A})\in K$, and the fact that the transformation
matrix $\mathcal{A}\vert_{x_0\mapsto 0}$ between the two homogeneous bases
of $P/\DF_W(I)$ is invertible implies $\det(\mathcal{A})\ne 0$.
Altogether, it follows that~$\overline{\mathcal{O}}$ is a
homogeneous $K[x_0]$-basis of~$\overline{P}/I\hom$, too. In particular,
the residue classes of~$\mathcal{O}$ form a $K$-basis of
$P/I\cong \overline{P}/(I\hom +(x_0-1))$, i.e.\ the ideal~$I$
has an $\mathcal{O}$-border basis.

For every $j\in\{1,\dots,\nu\}$, we have a representation
$b_j=\sum_{i=1}^\mu f_{ij} t_i + h_j$
with homogeneous polynomials $f_{ij}\in K[x_0]$ of degree $\deg_W(b_j)-\deg_W(t_i)$
and with a homogeneous polynomial $h_j\in I\hom$ of degree $\deg_W(b_j)$.
Setting $x_0\mapsto 1$ in this representation, we find
$g_j=b_j -\sum_{i=1}^\mu f_{ij}(1)\, t_i \in I$. It follows that these polynomials
form the $\mathcal{O}$-border basis of~$I$. By construction, we have
$b_j\in\Supp(\DF_W(g_j))$.

The first additional claim is a consequence of the observation that $\DF_W(G)$ is an
$\mathcal{O}$-border prebasis of~$\DF_W(I)$ and of~\cite{KR2}, Prop.~6.4.17.
To construct the desired flat family, we use the fact that~$G$ is a
Macaulay basis of~$I$ by what we have just shown
and conclude from~\cite{KR2}, 4.3.19 that $I\hom=(g_1\hom,\dots,
g_\nu\hom)$. From this the claim follows.
\end{proof}

Let us look at an example for this proposition.

\begin{example}
Consider the ideal $I=(-2x^2+xy-y^2-1,\, 8y^3+10x+9y)$
in the polynomial ring $P=\mathbb{Q}[x,y]$. The degree form ideal of~$I$
with respect to the standard grading, i.e.\ the grading defined by $W=(1\;1)$,
is $\DF_W(I)=(-2x^2 +xy-y^2,\, y^3)$. We want to use the order
ideal $\mathcal{O}=\{1,x,x^2,x^3,y,y^2\}$ whose border is given by
$\partial\mathcal{O} = \{xy, y^3,xy^2,x^2y,x^3y,x^4\}$.

\medskip
\makebox[11 true cm]{
\beginpicture
\setcoordinatesystem units <0.4cm,0.4cm>
\setplotarea x from 0 to 5, y from 0 to 4.1
\axis left /
\axis bottom /

\arrow <2mm> [.2,.67] from  4.5 0  to 5 0
\arrow <2mm> [.2,.67] from  0 3.6  to 0 4.1

\put {$\scriptstyle x^i$} [lt] <0.5mm,0.8mm> at 5.1 0
\put {$\scriptstyle y^j$} [rb] <1.7mm,0.7mm> at 0 4.1
\put {$\bullet$} at 0 0
\put {$\bullet$} at 1 0
\put {$\bullet$} at 0 1
\put {$\bullet$} at 2 0
\put {$\bullet$} at 0 2
\put {$\bullet$} at 3 0
\put {$\scriptstyle 1$} [lt] <-0.6mm,-1mm> at 0 0
\put {$\circ$} at 0 3
\put {$\circ$} at 1 2
\put {$\circ$} at 1 1
\put {$\circ$} at 2 1
\put {$\circ$} at 3 1
\put {$\circ$} at 4 0
\endpicture}
\medskip

It is easy to check that $\DF_W(I)$ has an $\mathcal{O}$-border
basis, namely $H=\{h_1,\dots,h_6\}$ with $h_1=xy-2x^2-y^2$, $h_2=y^3$,
$h_3=xy^2+4x^3$, $h_4=x^2y+2x^3$, $h_5=x^3y$, and $h_6=x^4$.
Therefore the proposition says that~$I$ has an $\mathcal{O}$-border basis
$G=\{g_1,\dots,g_6\}$, and that $h_i=\DF_W(g_i)$ for $i=1,\dots,6$.
Indeed, if we compute this border basis we find that it is
given by $g_1=xy-2x^2-y^2-1$, $g_2=y^3+\frac{5}{4}x+\frac{9}{8}y$,
$g_3=xy^2+4x^3+\frac{3}{4}x -\frac{1}{8}y$, $g_4=x^2y+2x^3 -\frac{1}{4}x
-\frac{1}{8}y$, $g_5=x^3y-\frac{1}{2}x^2-\frac{1}{8}y^2-\frac{3}{32}$,
and $g_6=x^4-\frac{1}{64}$.
\end{example}

An easy modification of this example shows that the converse implication
is not true without the hypothesis $b_j \in \Supp(\DF_W(g_j))$,
i.e.\ that an $\mathcal{O}$-border basis of~$I$ does not necessarily deform
to an $\mathcal{O}$-border basis of~$\DF_W(I)$.

\begin{example}\label{noDFdeform}
Consider the ideal $I=(x^2y,\, x^3 - \frac{1}{2}xy,\, xy^2,\, y^3)$
in $P=\mathbb{Q}[x,y]$. With respect to the standard
grading, we have $\DF_W(I)=( x^3,x^2y,xy^2,y^3)$. The ideal $\DF_W(I)$
does not have an $\mathcal{O}$-border basis for $\mathcal{O}=
\{1,x,x^2,x^3,y,y^2\}$. However, the ideal~$I$ has the $\mathcal{O}$-border
basis $G=\{g_1,\dots,g_6\}$, where $g_1=xy-2x^3$, $g_2=y^3$, $g_3=xy^2$, $g_4=x^2y$,
$g_5=x^3y$, and $g_6=x^4$.
\end{example}

The main reason why the last example exists
is that one of the terms in~$\mathcal{O}$ has a larger degree than
the term~$xy$ in the border of~$\mathcal{O}$. This suggests
the following notion.

\begin{definition}\label{DefMaxdeg}
Let~$P$ be graded by a matrix $W\in\Mat_{1,n}(\mathbb{N}_+)$.
The order ideal~$\mathcal{O}$ is said to have a {\it $maxdeg_W\!\!$ border}
if $\deg_W(b_j)\ge \deg_W(t_i)$ for $i=1,\dots,\mu$ and $j=1,\dots,\nu$.
In other words, no term in~$\mathcal{O}$ is allowed to have a degree
larger than any term in the border.
\end{definition}

Note that this condition is violated in Example~\ref{noDFdeform}.
By choosing suitable weights, many order ideals can be seen
to have a $\maxdeg_W$ border.

\begin{example}
Let $a\ge 1$, and let~$\mathcal{O}=\{1,x_1,x_1^2,\dots,x_1^a\}
\subset \mathbb{T}^n$. Then~$\mathcal{O}$ has a $\maxdeg_W$
border with respect to the grading given by $W=(1\;a\;\cdots\;a)$.
\end{example}

One consequence of an order ideal having a $\maxdeg_W$ border
is that $b_j \in \Supp(\LF_W(g_j))$ for $j=1,\dots,\nu$ and
every $\mathcal{O}$-border prebasis $G=\{g_1,\dots,g_\nu\}$.
Thus the proposition applies in particular
to order ideals having a $\maxdeg_W$ border.
Let us end this section with an example for this part of the
proposition.

\begin{example}\label{deftoDFex}
Let $\mathcal{O}=\{1,x,x^2,y,y^2\} \subset \mathbb T^2$.
Then we have  $\mathcal{O}=\mathbb{T}^2_{\le 2}\setminus \{ xy\}$, i.e.\
the order ideal~$\mathcal{O}$ has a $\maxdeg_W$ border with respect
to the standard grading.

Consider the ideal $I=(x^2+xy -\frac{1}{2}y^2-x-\frac{1}{2}y,\, y^3-y,\,
xy^2-xy)$ which is the vanishing ideal of the point set
$\mathbb X=\{(0,0),\, (0,-1),\, (1,0),\, (1,1),\, (-1,1)\}$ if ${\rm char}(K)\ne 2$.
We have $\partial\mathcal{O}= \{b_1,b_2,b_3,b_4,b_5\}$ with $b_1=x^3$, $b_2=x^2y$,
$b_3=xy$, $b_4=xy^2$, and $b_5=y^3$. The ideal~$I$ has an $\mathcal{O}$-border
basis, namely $G=\{g_1,g_2,g_3,g_4,g_5\}$ with $g_1=x^3-x$,
$g_2=x^2y-\frac{1}{2}y^2-\frac{1}{2}y$,
$g_3=xy+x^2-\frac{1}{2}y^2-x-\frac{1}{2}y$,
$g_4=xy^2+x^2-\frac{1}{2}y^2-x-\frac{1}{2}y$, and $g_5=y^3-y$.

The order ideal~$\mathcal{O}$ is not of the form $\mathcal{O}=\mathcal{O}_\sigma(I)$
for any term ordering~$\sigma$.
Using the proposition, we deform the border basis elements in~$G$ to their degree forms.
Thus the ideal $\DF_W(I)=(x^3,\, x^2y,\, xy+x^2-\frac{1}{2}y^2,\, xy^2,\, y^3)$
is a flat deformation of~$I$ and these five polynomials are an $\mathcal{O}$-border
basis of $\DF_W(I)$. The task of deforming the homogeneous ideal $\DF_W(I)$ further to the
border term ideal $\BT_{\mathcal{O}}=(x^3,\, x^2y,\, xy,\, xy^2,\, y^3)$
will be considered in Example~\ref{exdefcontinued}.
\end{example}

%
%

\section{The Border Basis Scheme}
\label{The Border Basis Scheme}

Let $\mathcal{O}=\{t_1,\dots,t_\mu\}$ be an order ideal in~$\mathbb T^n$, and let
$\partial\mathcal{O}=\{b_1,\dots,b_\nu\}$ be its border. In this section we define
a moduli space for {\it all}\/ zero-dimensional ideals having an $\mathcal{O}$-border
basis, and we use rational curves on this scheme to construct flat
deformations of border bases.

\begin{definition}\label{defBBS}
Let $\{c_{ij} \mid 1\le i\le \mu,\;
1\le j\le\nu\}$ be a set of further indeterminates.

\begin{items}
\item The {\em generic $\mathcal{O}$-border prebasis}
is the set of polynomials
$G=\{g_1,\dots,g_\nu\}$ in~$K[x_1,\dots,x_n,c_{11},\dots,c_{\mu\nu}]$
given by
$$
g_j = b_j -\sum_{i=1}^\mu c_{ij}t_i
$$

\item For $k=1,\dots,n$, let $\mathcal{A}_k \in\Mat_{\mu}(K[c_{ij}])$ be the
$k^{\rm th}$ formal multiplication matrix associated to~$G$ (cf.~\cite{KR2},
Def.\ 6.4.29). It is also called the $k^{\rm th}$ {\em generic multiplication matrix}\/
with respect to~$\mathcal{O}$.

\item The affine scheme $\mathbb{B}_{\mathcal{O}} \subseteq \mathbb{A}^{\mu\nu}$
defined by the ideal $I(\mathbb{B}_{\mathcal{O}})$ generated by the entries of the matrices
$\mathcal{A}_k \mathcal{A}_\ell -\mathcal{A}_\ell \mathcal{A}_k$ with
$1\le k<\ell\le n$ is called the {\em $\mathcal{O}$-border basis scheme}.

\item The coordinate ring $K[c_{11},\dots,c_{\mu\nu}]/I(\mathbb{B}_\mathcal{O})$
of the scheme $\mathbb{B}_{\mathcal{O}}$
will be denoted by~$B_{\mathcal{O}}$.

\end{items}
\end{definition}

By~\cite{KR2}, Thm.\ 6.4.30, a point $(\alpha_{ij})\in K^{\mu\nu}$ yields a
border basis $\sigma(G)$ when we apply the substitution $\sigma(c_{ij})=\alpha_{ij}$
to~$G$ if and only if $\sigma(\mathcal{A}_k)\,\sigma(\mathcal{A}_\ell)=
\sigma(\mathcal{A}_\ell)\,\sigma(\mathcal{A}_k)$
for $1\le k<\ell\le n$. Therefore the $K$-rational points of~$\mathbb
B_{\mathcal{O}}$ are in 1--1 correspondence with the $\mathcal{O}$-border bases of
zero-dimensional ideals in~$P$, and thus with all zero-dimensional ideals
having an $\mathcal{O}$-border basis.

\begin{remark}{\bf (Properties of Border Basis Schemes)}%
\label{BBSprops}\\
Currently, not much seems to be known about
border basis schemes. For instance, it is not clear which of them are
connected, reduced or irreducible. Here we collect some basic
observations.

\begin{items}
\item By definition, the ideal $I(\mathbb{B}_{\mathcal{O}})$ is generated
by polynomials of degree two.

\item The scheme $\mathbb{B}_{\mathcal{O}}$ can be embedded as an open
affine subscheme of the Hilbert scheme parametrizing subschemes
of~$\mathbb{A}^n$ of length~$\mu$ (see~\cite{MS}, Section 18.4).

\item There is an irreducible component of~$\mathbb{B}_{\mathcal{O}}$
of dimension $n\mu$ which is the closure of the set of radical ideals
having an $\mathcal{O}$-border basis.

\item The dimension of~$\mathbb{B}_{\mathcal{O}}$
is claimed to be $n\mu$ in~\cite{S}, Prop.\ 8.13.
Example~\ref{exIarrobino} shows that A.~Iarrobino's example of
a high-dimensional component of the Hilbert scheme yields a
counterexample to this claim.
It follows that the border basis scheme is in general not
irreducible.

\item For every term ordering~$\sigma$, there is a subset
of~$\mathbb{B}_{\mathcal{O}}$ which parametrized all
ideals~$I$ such that $\mathcal{O} = \mathcal{O}_\sigma(I)$.
These subsets have turned out to be useful for studying
the Hilbert scheme parametrizing subschemes
of~$\mathbb{A}^n$ of length~$\mu$ (see for instance~\cite{CV}
and~\cite{NS}).

\item In the case $n=2$ more precise information is available:
for instance, it is known that $\mathbb{B}_{\mathcal{O}}$ is reduced, irreducible
and smooth of dimension $2\mu$ (see~\cite{Ha}, \cite{Hu1} and~\cite{MS}, Ch.\ 18).

\end{items}
\end{remark}

As usual, a moduli space such as the border basis scheme comes
together with a universal family. In the present setting it is
defined as follows.

\begin{definition}
Let $G=\{g_1,\dots,g_\nu\} \subset K[x_1,\dots,x_n,c_{11},\dots,c_{\mu\nu}]$
with $g_j = b_j -\sum_{i=1}^\mu c_{ij}t_i$ for $j=1,\dots,\nu$ be the
generic $\mathcal{O}$-border prebasis.
The ring
$K[x_1,\dots,x_n,c_{11},\dots,c_{\mu\nu}]/(I(\mathbb{B}_{\mathcal{O}})+(g_1,
\dots, g_\nu ))$ will be denoted by~$U_{\mathcal{O}}$.
Then the natural homomorphism of $K$-algebras
$$
\Phi:\; B_{\mathcal{O}} \;\longrightarrow\; U_{\mathcal{O}} \cong
B_{\mathcal{O}}[x_1,\dots,x_n]/ (g_1,\dots,g_\nu)
$$
is called the {\em universal $\mathcal{O}$-border basis family}.
\end{definition}

The fibers of the universal $\mathcal{O}$-border basis family
are precisely the quotient rings $P/I$ for which~$I$ is a zero-dimensional
ideal which has an $\mathcal{O}$-border basis. The special fiber, i.e.\
the fiber corresponding to $(c_{11},\dots,c_{\mu\nu})$, is the ring $P/\BT_{\mathcal{O}}$.
It is the only fiber in the family which is defined by a monomial ideal.
Although it is known that
the universal family is free with basis~$\mathcal{O}$ (see~\cite{GLS}
or~\cite{Hu2}), we believe that the following proof which generalizes
the method in~\cite{M} is very elementary and conceptually simple.

\begin{theorem}{\bf (The Universal Border Basis Family)}\label{universal}\\
Let $\Phi: B_{\mathcal{O}} \longrightarrow U_{\mathcal{O}}$ be
the universal $\mathcal{O}$-border basis family. Then the residue
classes of the elements of~$\mathcal{O}$ are a
$B_{\mathcal{O}}$-module basis of~$U_{\mathcal{O}}$. In particular,
the map~$\Phi$ is a flat homomorphism.
\end{theorem}

\begin{proof}
First we prove that the residue classes $\overline{\mathcal{O}}$
are a system of generators of the
$B_{\mathcal{O}}$-module $U_{\mathcal{O}}\cong
B_{\mathcal{O}}[x_1,\dots,x_n]/(G)$
where $G=\{g_1, \dots, g_\nu\}$ is the generic $\mathcal{O}$-border prebasis.
In order to show that the map $\omega: B_{\mathcal{O}}^{\nu} \To
U_{\mathcal{O}}$ defined by $e_i\mapsto \bar t_i$ is surjective,
we may extend the base field and hence assume that~$K$ is algebraically
closed. By the local-global principle and the lemma of Nakayama,
it suffices to show that the induced map
$$
\bar\omega: (B_{\mathcal{O}})_{\mathfrak{m}}/\mathfrak{m}
(B_{\mathcal{O}}){\mathfrak{m}} \To
(B_{\mathcal{O}})_{\mathfrak{m}}[x_1,\dots,x_n]/((G) +
\mathfrak{m}(B_{\mathcal{O}})_{\mathfrak{m}}[x_1,\dots,x_n])
$$
is surjective for every maximal ideal $\mathfrak{m}=
(c_{ij}-\alpha_{ij})_{i,j}$ in~$B_{\mathcal{O}}$. In other words, we need to show
that the map~$\omega$ becomes surjective if we substitute values~$\alpha_{ij}\in K$
for the indeterminates~$c_{ij}$ and if these values have the property that
the maximal ideal $(c_{ij}-\alpha_{ij})_{i,j}$
contains $I(\mathbb{B}_{\mathcal{O}})$. Thus the claim follows from the fact that~$G$
becomes an $\mathcal{O}$-border basis after such a substitution, since
its associated formal multiplication matrices commute.

Now we show that~$\overline{\mathcal{O}}$ is $B_{\mathcal{O}}$-linearly independent.
We consider the free $B_{\mathcal{O}}$-submodule
$M=\bigoplus_{i=1}^{\mu} B_{\mathcal{O}}\,t_i$ of $B_{\mathcal{O}}[x_1,\dots,x_n]$
and proceed in the following manner.

\begin{enumerate}
\item We equip~$M$ with a suitable $B_{\mathcal{O}}[x_1,\dots,x_n]$-module structure.

\item We show that this $B_{\mathcal{O}}[x_1,\dots,x_n]$-module is cyclic and construct a
surjective $B_{\mathcal{O}}[x_1,\dots,x_n]$-linear map $\Theta: B_{\mathcal{O}}
[x_1,\dots,x_n] \To M$ which maps~$t_i$ to~$t_i$.

\item We prove that the kernel of~$\Theta$ is precisely $(G)$.
\end{enumerate}

Altogether, it follows that~$\Theta$ induces a map $\overline{\Theta}:
B_{\mathcal{O}}[x_1,\dots,x_n]/(G) \To M$ which is an isomorphism of
$B_{\mathcal{O}}$-modules and maps~$\bar t_i$ to~$t_i$. Thus $\overline{\mathcal{O}}=\{\bar t_1,
\dots,\bar t_{\mu}\}$ is a $B_{\mathcal{O}}$-basis of~$U_{\mathcal{O}}$, as claimed.

To do Step~1, we let $\overline{\mathcal{A}}_j$ be the image of the
generic multiplication matrix in $\Mat_{\mu}(B_{\mathcal{O}})$. Then we define
\begin{eqnarray}
a \ast \sum_{i=1}^\mu a_it_i & = &
(t_1,\dots,t_{\mu}) \cdot a\,\overline{\mathcal{I}}_\mu
\cdot (a_1,\dots,a_\mu)^{\rm tr}
= \sum_{i=1}^\mu a\, a_i\, t_i \\
x_j\ast \sum_{i=1}^\mu a_it_i & = &
(t_1,\dots,t_\mu) \cdot \overline{\mathcal{A}}_j
\cdot (a_1,\dots, a_\mu)^{\rm tr}
\end{eqnarray}
for $a, a_1,\dots,a_\mu\in B_{\mathcal{O}}$ and $j=1,\dots,n$.
Using this definition, the equalities
$$
x_k x_j\ast \sum_{i=1}^\mu a_it_i = x_k\ast (x_j\ast \sum_{i=1}^\mu a_it_i )=
(t_1,\dots,t_{\mu}) \cdot \overline{\mathcal{A}}_k\overline{\mathcal{A}}_j
\cdot (a_1,\dots,a_\mu)^{\rm tr}
\leqno{(3)}
$$
and the fact that the matrices~$\overline{\mathcal{A}}_j$ commute show that
this definition equips~$M$ with the structure of a $B_{\mathcal{O}}[x_1,\dots,x_n]$-module.
By using induction, we get
$$
f \ast \sum_{i=1}^\mu a_it_i = (t_1,\dots,t_\mu) \cdot
f(\overline{\mathcal{A}}_1,\dots,\overline{\mathcal{A}}_n) \cdot
(a_1,\dots, a_\mu)^{\rm tr} \leqno{(4)}
$$
for every $f\in B_{\mathcal{O}}[x_1,\dots,x_n]$ and all $a_1,\dots,a_\mu \in B_{\mathcal{O}}$.

For Step~2, we assume w.l.o.g.\ that $t_1=1$. Using induction on
$\deg(t_i)$, we want to show that $t_i\ast t_1=t_i$ for $i=1,\dots,\mu$.
The case $t_i=1$ follows from $(1)$. For the induction
step, we write $t_i=x_k t_\ell$ and using $(2)$, $(3)$ and $(4)$ we calculate
$$
t_i \ast t_1= x_k\ast (t_\ell\ast t_1) =x_k \ast t_\ell
= (t_1,\dots,t_\mu) \cdot \overline{\mathcal{A}}_k \cdot e_\ell^{\rm tr} =
(t_1,\dots,t_\mu)\cdot e_i^{\rm tr} = t_i
$$
It follows that $M$ is a cyclic $B_{\mathcal{O}}[x_1,\dots,x_n]$-module generated by~$t_1$.
Thus we obtain a surjective $B_{\mathcal{O}}[x_1,\dots,x_n]$-linear map
$\Theta: B_{\mathcal{O}}[x_1,\dots,x_n] \To M$ which is defined by~$f\mapsto f\ast t_1$.
We have just shown that~$\Theta$ satisfies $\Theta(t_i)=t_i$ for $i=1,\dots,\mu$.

Finally, to prove Step~3, we want to show that $\Theta(g_j)=0$ for $j=1,\dots,\nu$.
We write $b_j=x_kt_\ell$ and calculate $\Theta(g_i) = g_1\ast t_1 = (t_1,\dots, t_\mu) \cdot
g_j(\overline{\mathcal{A}}_1,\dots,\overline{\mathcal{A}}_n) \cdot e_1^{\rm tr}$.
In particular, we get
\begin{eqnarray*}
g_j(\overline{\mathcal{A}}_1,\dots,\overline{\mathcal{A}}_n) \cdot e_1^{\rm tr}
&=& b_j(\overline{\mathcal{A}}_1,\dots,\overline{\mathcal{A}}_n)\cdot e_1^{\rm tr} -
{\textstyle\sum\limits_{i=1}^\mu} c_{ij}\; t_i(\overline{\mathcal{A}}_1,\dots,
\overline{\mathcal{A}}_n) \cdot e_1^{\rm tr}\\
&=& \overline{\mathcal{A}}_k \cdot t_\ell(\overline{\mathcal{A}}_1,\dots,
\overline{\mathcal{A}}_n) \cdot e_1^{\rm tr} -
{\textstyle\sum\limits_{i=1}^\mu} c_{ij}\;e_i^{\rm tr}
= \overline{\mathcal{A}}_k \cdot e_\ell^{\rm tr} - {\textstyle\sum\limits_{i=1}^\mu}
c_{ij}\;e_i^{\rm tr}\\
&=&  {\textstyle\sum\limits_{i=1}^\mu} c_{ij}\; e_i^{\rm tr}
- {\textstyle\sum\limits_{i=1}^\mu} c_{ij}\; e_i^{\rm tr} =0
\end{eqnarray*}
We have checked that  $\Theta(g_j)=0$   for $j=1,\dots,\nu$.
Consequently, the map~$\Theta$ induces a $B_{\mathcal{O}}$-linear map
$\overline{\Theta}: B_{\mathcal{O}}[x_1,\dots,x_n]/( G) \To M$.
We know already that~$\overline{\mathcal{O}}$ generates the left-hand side
and $\mathcal{O}$ is a $B_{\mathcal{O}}$-basis of the right-hand side.
Hence the surjective map $\overline{\Theta}$ is also injective.
\end{proof}

In the remainder of this section we recall the connection between flat deformations
over~$K[z]$ of border bases and rational curves on the border basis scheme.
A rational curve on the $\mathcal{O}$-border basis scheme
corresponds to a $K$-algebra homomorphism $\Psi: B_{\mathcal{O}} \To K[z]$
of the corresponding affine coordinate rings. If we restrict the universal
family of $\mathcal{O}$-border bases to this rational curve, we obtain
the following flat deformation of border bases.

\begin{corollary}\label{ratcurve}
Let~$z$ be a new indeterminate, and let
$\Psi: B_{\mathcal{O}}\To K[z]$ be a
homomorphism of $K$-algebras. By applying the base change~$\Psi$ to the
universal family~$\Phi$, we get a homomorphism of $K[z]$-algebras
$$
\Phi_{K[z]}=\Phi\otimes_{B_{\mathcal{O}}} K[z]:\; K[z]
\To U_{\mathcal{O}}  \otimes_{B_{\mathcal{O}}} K[z]
$$
Then the residue classes of the elements of~$\mathcal{O}$ form a
$K[z]$-module basis of the right-hand side.
In particular, the map $\Phi_{K[z]}$ defines a flat family.
\end{corollary}

This corollary can be used to construct flat deformations over~$K[z]$ of border bases.
Suppose the maximal ideal $\Psi^{-1}(z-1)$ corresponds to a given
$\mathcal{O}$-border basis and the maximal ideal $\Psi^{-1}(z)$
is the ideal $( c_{11},\dots,c_{\mu\nu})$ which corresponds to the
border term ideal $( b_1,\dots,b_\nu )$. In other words, suppose
that the rational curve connects a given point to the point $(0,\dots,0)$
which corresponds to the border term ideal. Then the map $\Phi_{K[z]}$
defines a flat family over~$K[z]$ whose generic fiber $P/I$ is defined by the ideal~$I$
generated by the given $\mathcal{O}$-border basis and whose special fiber
$P/( b_1,\dots,b_\nu )$ is defined by the border term ideal.

Another application of the theorem is the following criterion for checking
the flatness of a family of border bases.

\begin{corollary}{\bf (Flatness Criterion for Families of Border
Bases)}\label{FlatCrit}\\
Let~$z$ be a new indeterminate, let $\widetilde{P}=K[z][x_1,\dots,x_n]$,
and let $g_j=b_j -\sum_{i=1}^\mu a_{ij}(z)t_i\in\widetilde{P}$
be polynomials with coefficients $a_{ij}(z)\in K[z]$.
Let $\widetilde{I}$ be the ideal in $\widetilde{P}$ generated by
$G=\{g_1,\dots,g_\nu\}$
and assume that the formal multiplication matrices $\mathcal{A}_k
\in\Mat_\mu(K[z])$ of~$G$ are pairwise commuting.

\begin{items}
\item For every $c\in K$, the set $\{g_1\vert_{z\mapsto
c},\dots, g_\nu\vert_{z\mapsto c}\}$ is an $\mathcal{O}$-border basis
of the ideal $I_c=\widetilde{I}\vert_{z\mapsto c}$.

\item The canonical $K$-algebra homomorphism
$$
\phi:\quad K[z] \;\To\; K[z][x_1,\dots,x_n]/\widetilde{I}
$$
defines a flat family. More precisely, the residue classes of the elements
of~$\mathcal{O}$ are a $K[z]$-basis of $K[z][x_1,\dots,x_n]/\widetilde{I}$.

\end{items}
\end{corollary}

\begin{proof}
First we show~a). For every $c\in K$, the matrices $\mathcal{A}_k\vert_{z\mapsto c}$
are the multiplication matrices of $G\vert_{z\mapsto c}$. Thus the claim follows
from~\cite{KR2}, 6.4.30.
Next we prove~b). Since the matrices $\mathcal{A}_k$ commute, the map
$B_{\mathcal{O}}\To K[z]$
defined by $c_{ij} \mapsto a_{ij}(z)$ is a well-defined
homomorphism of $K$-algebras. Hence it suffices to apply the preceding
corollary.
\end{proof}

\begin{remark}
If~$K$ is infinite, the hypothesis that the formal multiplication matrices
$\mathcal{A}_k$ commute can be replaced by the assumption that
the matrices $\mathcal{A}_k \vert_{z\mapsto c}$ commute for every
$c\in K$. This follows from the fact that a polynomial $f\in K[z]$ is zero
if and only if $f(c)=0$ for all $c\in K$.
\end{remark}

Let us have a look at one particular border basis scheme in detail.

\begin{example}\label{affinecell}
Consider the case $n=2$ and $\mathcal{O}=\{1,x,y,xy\}$. The border of~$\mathcal{O}$
is $\partial\mathcal{O} = \{y^2, x^2, xy^2, x^2y\}$, so that in our terminology
we have $\mu=4$, $\nu = 4$, $t_1 = 1$,  $t_2 = x$, $t_3 = y$, $t_4 = xy$,
$b_1 = y^2$, $b_2 = x^2$, $b_3 = xy^2$, and $b_4 = x^2y$.

The generic multiplication matrices are
$$
\mathcal{A}_x =
\left( \begin{array}{cccc}
0 & c_{1\, 2\, } & 0 & c_{1\, 4\, } \\
1 & c_{2\, 2\, } & 0 & c_{2\, 4\, } \\
0 & c_{3\, 2\, } & 0 & c_{3\, 4\, } \\
0 & c_{4\, 2\, } & 1 & c_{4\, 4\, } \end{array}\right)
\hbox{\quad and \quad }
\mathcal{A}_y=
\left( \begin{array}{cccc}
0 & 0 & c_{1\, 1\, } & c_{1\, 3\, } \\
0 & 0 & c_{2\, 1\, } & c_{2\, 3\, } \\
1 & 0 & c_{3\, 1\, } & c_{3\, 3\, } \\
0 & 1 & c_{4\, 1\, } & c_{4\, 3\, } \end{array}\right)
$$

When we compute the ideal generated by the entries of
$\mathcal{A}_x \mathcal{A}_y -\mathcal{A}_y \mathcal{A}_x$ and
simplify its system of generators, we see that
the ideal $I(\mathbb{B}_{\mathcal{O}})$ is generated by
$$
\left.\begin{array}{l}
 \{
 c_{23}c_{41}c_{42} - c_{21}c_{42}c_{43} + c_{21}c_{44}+ c_{11} - c_{23},\;\;
  -c_{21}c_{32} - c_{34}c_{41} + c_{33},\\
\;c_{34}c_{41}c_{42}- c_{32}c_{41}c_{44}+ c_{32}c_{43}+ c_{12}- c_{34},\;\;
-c_{21}c_{32}- c_{23}c_{42}+ c_{24}, \\
\; -c_{23}c_{32}c_{41}+ c_{21}c_{32}c_{43} - c_{21}c_{34}+ c_{13},\;\;
 c_{21}c_{42} + c_{41}c_{44} + c_{31}- c_{43}, \\
\;-c_{21}c_{34}c_{42}+ c_{21}c_{32}c_{44} - c_{23}c_{32} + c_{14},\;\;
 c_{32}c_{41}+ c_{42}c_{43} + c_{22} - c_{44}
 \}
 \end{array}\right.
$$
Thus there are eight free indeterminates, namely
$c_{21}$, $c_{23}$, $c_{32}$, $c_{34}$, $c_{41}$, $c_{42}$, $c_{43}$,
and~$c_{44}$, while the remaining indeterminates depend on the free ones
by the polynomial expressions above. From this we conclude that the border
basis scheme $\mathbb{B}_{\mathcal{O}}$ is an {\it affine cell}
of the corresponding Hilbert scheme,
i.e.\ an open subset which is isomorphic to an affine space.
(This result is in agreement with~\cite{Hu1}, Thm.\ 7.4.1,
but not with~\cite{MS}, Example 18.6.)

Its coordinate ring is explicitly represented by the isomorphism
$$B_{\mathcal{O}} \;\longiso\;
K[c_{21}, c_{23}, c_{32}, c_{34}, c_{41}, c_{42}, c_{43}, c_{44}]
$$
given by
$$
\left.\begin{array}{l}
c_{11} \;\longmapsto\;  -c_{23}c_{41}c_{42} + c_{21}c_{42}c_{43} - c_{21}c_{44}+  c_{23}\\
c_{12} \;\longmapsto\;  -c_{34}c_{41}c_{42} + c_{32}c_{41}c_{44} - c_{32}c_{43}+ c_{34} \\
c_{13} \;\longmapsto\;  c_{23}c_{32}c_{41} -  c_{21}c_{32}c_{43} + c_{21}c_{34}\\
c_{14} \;\longmapsto\;  c_{21}c_{34}c_{42} - c_{21}c_{32}c_{44} + c_{23}c_{32}\\
c_{22} \;\longmapsto\;  -c_{32}c_{41} -  c_{42}c_{43} +  c_{44}\\
c_{24} \;\longmapsto\;   c_{21}c_{32}  + c_{23}c_{42}\\
c_{31} \;\longmapsto\;  -c_{21}c_{42} - c_{41}c_{44} + c_{43}\\
c_{33} \;\longmapsto\;    c_{21}c_{32} + c_{34}c_{41}
 \end{array}\right.
$$
Hence we have
$U_{\mathcal{O}} \cong
K[x,y, c_{21}, c_{23}, c_{32}, c_{34}, c_{41}, c_{42}, c_{43}, c_{44}]/
(\widetilde{g}_1, \widetilde{g}_2, \widetilde{g}_3,  \widetilde{g}_4)$
where
\begin{eqnarray*}
\widetilde{g}_1 &=& y^2 - (-c_{23}c_{41}c_{42} + c_{21}c_{42}c_{43} - c_{21}c_{44}+  c_{23})
\\ && - c_{21}x - (-c_{21}c_{42} - c_{41}c_{44} + c_{43})y
- c_{41}xy,\\
\widetilde{g}_2 &=& x^2 - (-c_{34}c_{41}c_{42} + c_{32}c_{41}c_{44} - c_{32}c_{43}+ c_{34})
\\ && - (-c_{32}c_{41} -  c_{42}c_{43} +  c_{44})x - c_{32}y
- c_{42}xy,\\
\widetilde{g}_3 &=& xy^2 - (c_{23}c_{32}c_{41} -  c_{21}c_{32}c_{43} + c_{21}c_{34})
\\ && - c_{23}x - (c_{21}c_{32} + c_{34}c_{41})y - c_{43}xy,\\
\widetilde{g}_4 &=& x^2y -(c_{21}c_{34}c_{42} - c_{21}c_{32}c_{44} + c_{23}c_{32})
\\ && - (c_{21}c_{32}  + c_{23}c_{42})x - c_{34}y - c_{44}xy,\\
\end{eqnarray*}

The ideal $(\widetilde{g}_1, \widetilde{g}_2, \widetilde{g}_3,  \widetilde{g}_4)$
is the defining ideal of the family of all subschemes of
length four of the affine plane
which have the property that their coordinate ring admits $\overline{\mathcal{O}}$
as a vector space basis.
Since the border basis scheme is isomorphic to an affine space in this case,
we can connect every point to the point corresponding to $(x^2,y^2)$ by a
rational curve. Therefore every ideal in the family can be deformed by a flat deformation
to the monomial ideal $(x^2, y^2)$. Algebraically, it suffices to substitute each
free indeterminate $c_{ij}$ with $z c_{ij}$ where~$z$ is a new indeterminate.
We get the $K$-algebra homomorphism
$$
\Phi_{K[z]}: K[z] \To K[x,y, z, c_{21}, c_{23}, c_{32}, c_{34}, c_{41}, c_{42},
c_{43}, c_{44}]/ (\overline{g}_1, \overline{g}_2, \overline{g}_3,  \overline{g}_4)
$$
where
\begin{eqnarray*}
\overline{g}_1 &=& y^2 - (-z^3c_{23}c_{41}c_{42} + z^3c_{21}c_{42}c_{43}
- z^2c_{21}c_{44}+  zc_{23})
\\ && - zc_{21}x - (-z^2c_{21}c_{42} - z^2c_{41}c_{44} +z c_{43})y - zc_{41}xy,\\
\overline{g}_2 &=& x^2 - (-z^3c_{34}c_{41}c_{42} + z^3c_{32}c_{41}c_{44}
- z^2c_{32}c_{43}+ zc_{34})
\\ && - (-z^2c_{32}c_{41} -  z^2c_{42}c_{43} + z c_{44})x - zc_{32}y - zc_{42}xy,\\
\overline{g}_3 &=& xy^2 - (z^3c_{23}c_{32}c_{41} -  z^3c_{21}c_{32}c_{43} + z^2c_{21}c_{34})
\\ && - zc_{23}x - (z^2c_{21}c_{32} +z^2 c_{34}c_{41})y - zc_{43}xy,\\
\overline{g}_4 &=& x^2y -(z^3c_{21}c_{34}c_{42} - z^3c_{21}c_{32}c_{44} + z^2c_{23}c_{32})
\\ && - (z^2c_{21}c_{32}  + z^2c_{23}c_{42})x - zc_{34}y - zc_{44}xy,\\
\end{eqnarray*}
By Corollary~\ref{ratcurve}, this homomorphism is flat.
For every point on the border basis scheme,
it connects the corresponding ideal to
$\BT_{\mathcal{O}}=(y^2, x^2, xy^2, x^2y) = (x^2,y^2)$.
\end{example}

The next example shows that natural families of ideals can lead us
out of the affine open subset~$\mathbb{B}_{\mathcal{O}}$ of the Hilbert scheme.

\begin{example}
Using $K=\mathbb{R}$ and $P=\mathbb{R}[x,y]$, we consider the family of reduced
zero-dimensional schemes $\mathbb{X}_a = \{(a,2),\, (0,1),\, (0,0),\, (1,0)\}
\subset \mathbb{R}^2$ with $a\in\mathbb{R}$.

\medskip
\makebox[11 true cm]
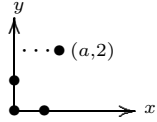
{
\beginpicture
\setcoordinatesystem units <0.4cm,0.4cm>
\setplotarea x from 0 to 4, y from 0 to 3.1
\axis left /
\axis bottom /

\arrow <2mm> [.2,.67] from  3.5 0  to 4 0
\arrow <2mm> [.2,.67] from  0 2.6  to 0 3.1

\put {$\scriptstyle x$} [lt] <0.5mm,0.8mm> at 4.1 0
\put {$\scriptstyle y$} [rb] <1.7mm,0.7mm> at 0 3.1
\put {$\bullet$} at 0 0
\put {$\bullet$} at 1 0
\put {$\bullet$} at 0 1
\put {$\bullet$} at 1.5 2
\put {$\scriptstyle (a,2)$} at 2.6 2
\put {$\cdots$} at 0.8 2
\endpicture}
\medskip

For $\sigma={\tt DegRevLex}$, the reduced $\sigma$-Gr\"obner basis of
the vanishing ideal $I_a\subset P$ of~$\mathbb{X}_a$ is
$$
G'_a=\{ x^2+\tfrac{1}{2}\,a(1-a)y^2-x -\tfrac{1}{2}\,a(1-a)\,y,\; xy-ay^2+ay,\;
y^3-3y^2+2y \}
$$
and thus we have $\mathcal{O}_\sigma(I_a)=\{1,x,y,y^2\}$. We may extend~$G'_a$
to an $\mathcal{O}_\sigma(I_a)$-border basis of~$I_a$ and get
$$
G_a=G'_a \;\cup\; \{ xy^2-2ay^2 +2ay  \}
$$
The residue classes of the elements of~$\mathcal{O}_\sigma(I)$
are a vector space basis of $P/I_a$ for every $a\in\mathbb{R}$.
We let $I=( x^2+\tfrac{1}{2}\,z(1-z)y^2-x -\tfrac{1}{2}\,z(1-z)y,\,
xy-z y^2+zy,\, y^3-3y^2+2y,\, xy^2-2zy^2 +2zy ) \subset P[z]$.
Then the natural map $\mathbb{R}[z]\To P[z]/I$ is a flat
homomorphism whose fibers are the rings $P/I_a$. Thus the
point corresponding to~$G_a$ on the border basis scheme
$\mathbb{B}_{\mathcal{O}_\sigma(I_a)}$ is connected to the point
representing~$G_0$ via a rational curve.

Now we consider the order ideal $\mathcal{O}=\{1,x,y,xy\}$. For $a\ne 0$, the
set~$\mathbb X_a$ is a complete intersection of type $(2,2)$. Its vanishing
ideal~$I_a$ has an $\mathcal{O}$-border basis, namely
$$
H_a= \{ y^2-\tfrac{1}{a}\,xy-y,\;  xy^2-2xy,\; x^2y-axy,\;
x^2+\tfrac{1}{2}\,(1-a)xy-x \}
$$
However, for $a=0$, the ideal $I_0$ has no $\mathcal{O}$-border basis
because $xy\in I_0$. One of the coefficients in~$H_a$ tends to~$\infty$
as $a\To 0$. This happens since the scheme~$\mathbb{B}_{\mathcal{O}}$
is not complete.
\end{example}

\bigbreak
%
%

\section{Defining Equations for the Border Basis Scheme}
\label{Defining Equations for the Border Basis Scheme}

The defining equations for the border basis scheme
can be constructed in different ways.
One construction is given
by imposing the commutativity law to the multiplication matrices,
as we have seen in the preceding section.
Another construction was given in~\cite{Hu2}, and a different
but related one in~\cite{KK1} and~\cite{S}.
After describing this alternative construction, we use it to get
rid of as many generators of the vanishing ideal of~$\mathbb{B}_{\mathcal{O}}$
as possible and examine some claims in~\cite{S} in this regard.

Let $\mathcal{O}=\{t_1,\dots,t_\mu\}$ be an order ideal and
$\partial\mathcal{O}=\{b_1,\dots,b_\nu\}$ its border. In~\cite{KK1},
Def.~17, two terms $b_i,b_j\in\partial\mathcal{O}$ are called
{\it next-door neighbors} if $b_i=x_k b_j$ for some $k\in\{1,\dots,n\}$
and {\it across-the street neighbors} if $x_k b_i = x_\ell b_j$
for some $k,\ell\in\{1,\dots,n\}$. In addition to these notions
we shall say that across-the-street neighbors $b_i,b_j$
with $x_k b_i= x_\ell b_j$ are {\it across-the-corner neighbors}
of there exists a term $b_m\in\partial\mathcal{O}$ such that
$b_i=x_\ell b_m$ and $b_j=x_k b_m$.

In~\cite{S}, Def.~8.5, the graph
whose vertices are the border terms and whose edges are given by
the neighbor relation is called the {\it border web} of~$\mathcal{O}$.
The Buchberger criterion for border bases (see~\cite{KK1}, Prop.~18
and~\cite{S}, Thm.~8.11) says that an $\mathcal{O}$-border
prebasis $\{g_1,\dots,g_\nu\}$ with $g_j=b_j-\sum_{i=1}^\mu
a_{ij}t_i$ and $a_{ij}\in K$ is an $\mathcal{O}$-border basis
if and only if the S-polynomials $S(g_i,g_j)$ reduce to zero
using~$G$ for all $(i,j)$ such that~$b_i$ and~$b_j$ are neighbors.
This characterization can be used to construct the equations
defining the border basis scheme in an alternative way.

\begin{proposition}{\bf (Lifting Neighbor Syzygies)}%
\label{altgenBBS}\\
Let $G=\{g_1,\dots,g_\nu\}$ be the generic $\mathcal{O}$-border
prebasis, where $g_j=b_j -\sum_{i=1}^\mu c_{ij}t_i \in K[x_1,\dots,x_n,
c_{11},\dots, c_{\mu\nu}]$, let $\mathcal{A}_1,\dots,\mathcal{A}_n
\in\Mat_{\mu}(K[c_{ij}])$ be the generic multiplication matrices
with respect to~$\mathcal{O}$, and let $c_j=(c_{1j},\dots,c_{\mu j})^{\rm tr}
\in\Mat_{\mu,1}(K[c_{ij}])$ for $j=1,\dots,\nu$.
Consider the following sets of polynomials
in $K[c_{11},\dots,c_{\mu\nu}]$:

\begin{enumerate}
\item If $b_i,b_j\in\partial\mathcal{O}$ are next-door neighbors
with $b_i=x_k b_j$, let $\ND(i,j)$ be the set of polynomial entries
of $c_i - \mathcal{A}_k c_j$.

\item If $b_i,b_j\in\partial\mathcal{O}$ are across-the-street neighbors
with $x_k b_i=x_\ell b_j$, let $\AS(i,j)$ be the set of polynomial entries
of $\mathcal{A}_k c_i - \mathcal{A}_\ell c_j$.
\end{enumerate}

Then the following claims hold true.

\begin{items}
\item The union of all sets $\ND(i,j)$ and all sets $\AS(i,j)$ contains
the set of the nontrivial entries of the commutators $\mathcal{A}_k
\mathcal{A}_\ell -\mathcal{A}_\ell \mathcal{A}_k$ with $1\le k<\ell \le n$.

\item If one removes from this union all sets $\AS(i,j)$ such that
$b_i,b_j$ are across-the-corner neighbors, one gets precisely
the set of the nontrivial entries of the commutators $\mathcal{A}_k
\mathcal{A}_\ell -\mathcal{A}_\ell \mathcal{A}_k$ with $1\le k<\ell \le n$.
In particular, the remaining union generates the vanishing ideal
$I(\mathbb{B}_{\mathcal{O}})$ of the $\mathcal{O}$-border basis scheme.

\item The polynomials in the sets $\AS(i,j)$ corresponding to
across-the-corner neighbors $b_i,b_j$ are contained
in~$I(\mathbb{B}_{\mathcal{O}})$.

\end{items}
\end{proposition}

\begin{proof} First we prove~a) and~b).
The S-polynomials $g_i - x_k g_j$ resp.\ $x_k g_i - x_\ell g_j$
are $K[c_{ij}]$-linear combinations of terms in $\mathcal{O}\cup
\partial\mathcal{O}$. We want to find representations of these
polynomials as $K[c_{ij}]$-linear combinations of elements
of~$\mathcal{O}$ only. Since we have $b_i - x_k b_j=0$ resp.\
$x_k b_i -x_\ell b_j=0$, we have to represent
$(-\sum_{m=1}^\mu c_{mi}t_m) - x_k \, (-\sum_{m=1}^\mu c_{mj}t_m)$
resp.\ $x_k \, (-\sum_{m=1}^\mu c_{mi}t_m) - x_\ell \, (-\sum_{m=1}^\mu c_{mj}t_m)$
using~$\mathcal{O}$. By the definition of the generic multiplication matrices,
these representations are given by $(t_1,\dots,t_\mu) \cdot
(c_i - \mathcal{A}_k c_j)$ resp.\ $(t_1,\dots,t_\mu) \cdot
(\mathcal{A}_k c_i - \mathcal{A}_\ell c_j)$. The coefficients of the terms~$t_i$
in these representations are precisely the polynomials in $\ND(i,j)$ resp.\ in
$\AS(i,j)$.

Now we consider the polynomials in the sets $\ND(i,j)$ and in the sets
$\AS(i,j)$ for which $b_i,b_j$ are not across-the-corner neighbors. The fact that
these polynomials are exactly the nontrivial entries of the
commutators $\mathcal{A}_k \mathcal{A}_\ell -\mathcal{A}_\ell \mathcal{A}_k$
was checked in~\cite{KK1}, Section~4 resp.~\cite{S}, Prop.~8.10.

It remains to show~c). Let $b_i=x_\ell b_m$ and $b_j=x_k b_m$.
By what we have shown so far, the polynomials which
are the components of $c_i -\mathcal{A}_\ell c_m$ and $c_j -\mathcal{A}_k c_m$
are contained in~$I(\mathbb{B}_{\mathcal{O}})$. Moreover, the polynomial
entries of $\mathcal{A}_k \mathcal{A}_\ell -\mathcal{A}_\ell\mathcal{A}_k$
are in~$I(\mathbb{B}_{\mathcal{O}})$. Therefore also the components of
$$
\mathcal{A}_k c_i -\mathcal{A}_\ell c_j = \mathcal{A}_k (c_i - \mathcal{A}_\ell c_m)
+ (\mathcal{A}_k \mathcal{A}_\ell - \mathcal{A}_\ell\mathcal{A}_k) c_m
-\mathcal{A}_\ell( c_j - \mathcal{A}_k c_m)
$$
are contained in~$I(\mathbb{B}_{\mathcal{O}})$. These components are exactly
the polynomials in~$\AS(i,j)$.
\end{proof}

Another way of phrasing this proposition is to say that, for~$G$ to be
a border basis, the neighbor syzygies
$e_i -x_k e_j$ resp.\ $x_k e_i -x_\ell e_j$ of the border tuple
$(b_1,\dots,b_\nu)$ have to lift to syzygies of $(g_1,\dots,g_\nu)$
and that the defining equations of~$\mathbb{\mathcal{O}}$ are precisely
the equations expressing the existence of these liftings (see~\cite{KK1},
Ex.~23). Now it is a well-known phenomenon in Gr\"obner basis theory
that it suffices to lift a minimal set of generators of the syzygy module
of the leading terms (see for instance~\cite{KR1}, Prop.~2.3.10).
In~\cite{S}, Props.~8.14 and~8.15, an attempt was made to
use a similar idea for removing unnecessary generators
of~$I(\mathbb{B}_{\mathcal{O}})$. However, the claims made there are
not correct in general, as the following examples show.

The first example has surfaced in a number of different contexts,
see the papers~\cite{K}, \cite{L} and the references therein.

\begin{example}\label{exa1xyz}
Let us consider $P=K[x,y,z]$ and $\mathcal{O}=\{1,x,y,z\}$.
The border $\partial\mathcal{O}=\{b_1,\dots,b_6\}$
with $b_1=x^2$, $b_2=xy$, $b_3=xz$, $b_4=y^2$, $b_5=yz$, and $b_6=z^2$
has a very simple border web consisting of nine across-the-street neighbors:

\medskip
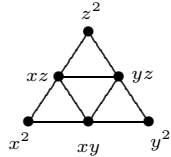
\makebox[11 true cm]{
\beginpicture
\setcoordinatesystem units <0.4cm,0.4cm>
\setplotarea x from -0.5 to 4.5, y from -0.9 to 3.5

\put {$\bullet$} at 0 0
\put {$\bullet$} at 2 0
\put {$\bullet$} at 4 0
\put {$\bullet$} at 1 1.5
\put {$\bullet$} at 3 1.5
\put {$\bullet$} at 2 3
\put {$\scriptstyle x^2$}  at -0.3 -0.6
\put {$\scriptstyle y^2$}  at 4.4 -0.6
\put {$\scriptstyle z^2$}  at 2.1 3.7
\put {$\scriptstyle xy$}   at 2 -0.9
\put {$\scriptstyle xz$}  at 0.3 1.5
\put {$\scriptstyle yz$}  at 3.8 1.5
\setlinear
\putrule from 0 0 to 4 0
\putrule from 1 1.5 to 2.8 1.5
\plot
0 0 %
2 3 %
4 0 %
/
\plot
1 1.5 %
2 0 %
3 1.5 %
/

\endpicture}
\medskip

These across-the-street neighbors yield $9\cdot 4 = 36$ quadratic
equations for $I(\mathbb{B}_{\mathcal{O}})$ in $K[c_{11},\dots,c_{46}]$.
Contrary to the claim in~\cite{S}, Prop.\ 8.15, the equations
for the neighbor pair $(x^2,xy)$ are not contained in the ideal
generated by the remaining 32 equations. In fact, in agreement with
Proposition~\ref{removing}, it turns out that the four equations
corresponding to the pair $(xy,xz)$ are contained in the ideal generated
by the eight equations corresponding to the two pairs $(xy,yz)$ and $(xz,yz)$
(see Example~\ref{exa1xyzcont}).

In order to see whether the ideal $I(\mathbb{B}_{\mathcal{O}})$ is a complete
intersection (as claimed in~\cite{S}, p.\ 297), we examine its generators more closely.
If we define a grading by letting
$\deg_W(c_{1j})=2$ for $j=1,\dots,6$ and $\deg_W(c_{ij})=1$
for $i>1$, the 36 generators are homogeneous with respect to the grading
given by~$W$. Every minimal system of generators of the ideal $I(\mathbb{B}_{\mathcal{O}})$
consists of~21 polynomials, while its height is~12.
Hence it is very far from being a complete intersection.

The indeterminates $c_{11},\dots,c_{16}$ corresponding
to the constant coefficients of the generic border basis form
the linear parts of six of the 21 minimal generators and do not divide
any of the other terms. We may eliminate them and
obtain an ideal~$J$ in~$Q=K[c_{21},\dots,c_{46}]$ which has (after interreduction)
15 homogeneous quadratic generators. Geometrically speaking, there is a projection
to an 18-dimensional affine space which maps the border basis scheme
isomorphically to a homogeneous subscheme of~$\mathbb{A}^{18}$.
In fact, it is known that this scheme is an affine cone with 3-dimensional
vertex over the Grassmannian ${\rm Grass}(2,6)\subset \mathbb{P}^{14}$.

The ideal~$J$ is prime and the ring~$Q/J$ is Gorenstein with Hilbert
series $(1+6z+6z^2+1)/(1-z)^{12}$. The minimal number of generators of~$J$ is~15.
The border basis scheme is irreducible and has the expected dimension, namely~12.
\end{example}

Also the lifting of trivial syzygies fails in the border basis scheme
setting, as our next example shows (see also Example~\ref{affinecell}).

\begin{example}\label{liftfails}
Let $P=K[x,y]$ and $\mathcal{O}=\{1,x,y,xy\}$. Then the border of~$\mathcal{O}$
is $\partial\mathcal{O}=\{x^2,y^2,x^2y,xy^2\}$. It has two next-door neighbors
$(x^2,x^2y),\ (y^2,xy^2)$ and one across-the-street neighbor $(x^2y,xy^2)$.
If one includes the ``trivial syzygy pair'' $(x^2,y^2)$, there is one loop
in the border web:

\medskip
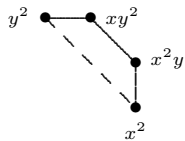
\makebox[11 true cm]{
\beginpicture
\setcoordinatesystem units <0.6cm,0.6cm>
\setplotarea x from -0.5 to 2.5, y from -0.5 to 2.5

\put {$\bullet$} at 2 0
\put {$\bullet$} at 2 1
\put {$\bullet$} at 1 2
\put {$\bullet$} at 0 2
\put {$\scriptstyle x^2$}  at 2 -0.5
\put {$\scriptstyle y^2$}  at -0.6 2
\put {$\scriptstyle x^2y$}   at 2.7 1.1
\put {$\scriptstyle xy^2$}  at 1.7 2
\setlinear
\plot
0 2 %
1 2 %
2 1 %
2 0 %
/
\setdashes
\plot
0 2 %
2 0 %
/

\endpicture}
\medskip

The neighbor pairs yield four equations each for the defining
ideal of~$\mathbb{B}_{\mathcal{O}}$. Contrary to a claim
in~\cite{S}, p.\ 297, one cannot drop one of these
sets of four polynomials without changing the ideal.
Thus the lifting of a ``trivial'' syzygy cannot be used to remove
defining equations for the border basis scheme.

Interestingly, in the case at hand, the ideal $I(\mathbb{B}_{\mathcal{O}})$
is indeed a complete intersection: there exists a subset of~8 of the
12~equations which generates~$I(\mathbb{B}_{\mathcal{O}})$ minimally and
$\dim(K[c_{11},\dots,c_{44}]/I(\mathbb{B}_{\mathcal{O}}))=8$.
But the unnecessary generators are spread around the blocks
coming from the neighbor pairs.
\end{example}

Our next example shows how one can sometimes get rid of some generators
of~$I(\mathbb{B}_{\mathcal{O}})$ using part~c) of the proposition.

\begin{example}
Consider $P=K[x,y]$ and $\mathcal{O}=\{1,x,y,x^2,y^2\}$.
Then we have $\partial\mathcal{O}=\{b_1,\dots,b_5\}$ with $b_1=y^3$,
$b_2=xy^2$, $b_3=xy$, $b_4=x^2y$, and $b_5=x^3$,
two next-door neighbors $(xy,xy^2)$ and $(xy,x^2y)$, two proper
across-the-street neighbors $(y^3,xy^2)$ and $(x^2y,x^3)$, and
one pair of across-the-corner neighbors $(xy^2,x^2y)$.
Thus the border web of~$\mathcal{O}$ looks as follows.

\medskip
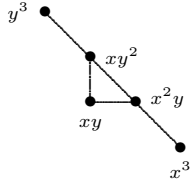
\makebox[11 true cm]{
\beginpicture
\setcoordinatesystem units <0.6cm,0.6cm>
\setplotarea x from -0.5 to 3.5, y from -0.5 to 3.5

\put {$\bullet$} at 3 0
\put {$\bullet$} at 2 1
\put {$\bullet$} at 1 1
\put {$\bullet$} at 1 2
\put {$\bullet$} at 0 3
\put {$\scriptstyle x^3$}  at 3 -0.5
\put {$\scriptstyle y^3$}  at -0.6 3
\put {$\scriptstyle x^2y$}   at 2.7 1.1
\put {$\scriptstyle xy^2$}  at 1.7 2
\put {$\scriptstyle xy$}  at 1 0.5
\setlinear
\plot
0 3 %
1 2 %
1 1 %
2 1 %
3 0 %
/
\plot
1 2 %
1.5 1.5 %
2 1 %
/
\endpicture}
\medskip

Using part~c) of the proposition, we know that
~$I(\mathbb{B}_{\mathcal{O}})$ is generated by $\AS(1,2)$, $\AS(4,5)$,
$\ND(2,3)$, and $\ND(3,4)$. In fact, using \cocoa, we may check that
none of these sets can be removed without changing the ideal.
\end{example}

On the positive side, the following proposition allows us to
remove at least a few polynomials from the system of generators
of~$I(\mathbb{B}_{\mathcal{O}})$ given in Proposition~\ref{altgenBBS}.

\begin{proposition}{\bf (Removing Redundant Generators of
$I(\mathbb{B}_{\mathcal{O}})$)}\label{removing}\\
Let $\mathcal{O}=\{t_1,\dots,t_\mu\}$ be an order ideal with border
$\partial\mathcal{O}= \{b_1,\dots,b_\nu\}$, and
let~$H$ be a system of generators of~$I(\mathbb{B}_{\mathcal{O}})$.

\begin{items}
\item Suppose that there exist $i,j,k \in\{1,\dots,\nu\}$ and
$\ell,m\in\{1,\dots,n\}$ such that $b_k=x_\ell b_i = x_m b_j$.
If the sets $\AS(i,j)$, $\ND(i,k)$ and $\ND(j,k)$ are contained in~$H$
and one removes one of these sets, the remaining polynomials
still generate~$I(\mathbb{B}_{\mathcal{O}})$.

\item Suppose that there exist $i,j,k \in\{1,\dots,\nu\}$ and
$\alpha,\beta,\gamma\in \{1,\dots,n\}$ such that $x_\alpha b_i =
x_\beta b_j = x_\gamma b_k$. If the sets $\AS(i,j)$, $\AS(i,k)$ and
$\AS(j,k)$ are contained in~$H$ and one removes one of these sets,
the remaining polynomials still generate~$I(\mathbb{B}_{\mathcal{O}})$.

\end{items}
\end{proposition}

\begin{proof} Let $\mathcal{A}_1\dots,\mathcal{A}_n$ be the generic
multiplication matrices with respect to~$\mathcal{O}$, and let
$c_j=(c_{1j},\dots,c_{\mu j})^{\rm tr}
\in\Mat_{\mu,1}(K[c_{ij}])$ for $j=1,\dots,\nu$.

First we show~a). The polynomials in $\AS(i,j)$ are the components of
$\mathcal{A}_\ell \cdot c_i-\mathcal{A}_m \cdot c_j$, the polynomials in
$\ND(i,k)$ are the components of $c_k-\mathcal{A}_\ell \cdot c_i$,
and the polynomials in $\ND(j,k)$ are the components of $c_k-\mathcal{A}_m
\cdot c_j$. Thus the claim follows from
$$
(\mathcal{A}_\ell \cdot c_i-\mathcal{A}_m \cdot c_j) +
(c_k-\mathcal{A}_\ell \cdot c_i) - (c_k-\mathcal{A}_m \cdot c_j) =0
$$

To show~b), we argue similarly. The polynomials in $\AS(i,j)$ are the components of
$\mathcal{A}_\alpha \cdot c_i - \mathcal{A}_\beta \cdot c_j$, the polynomials in
$\AS(i,k)$ are the components of $\mathcal{A}_\alpha \cdot c_i -
\mathcal{A}_\gamma \cdot c_k$, and the polynomials in $\AS(j,k)$ are the components
of $\mathcal{A}_\beta \cdot c_j - \mathcal{A}_\gamma \cdot c_k$.
\end{proof}

Let us illustrate the application of this proposition with a couple of
examples.

\begin{example}
Let $P=K[x,y,z]$ and $\mathcal{O}=[1,x,y,z,xy]$. Then we have
$\partial\mathcal{O}=\{b_1,\dots,b_8\}$ with $b_1=z^2$, $b_2=yz$,
$b_3=xz$, $b_4=y^2$, $b_5=x^2$, $b_6=xyz$, $b_7=xy^2$, and $b_8=x^2y$.
There are four next-door neighbors $(yz,xyz)$, $(xz,xyz)$, $y^2,xy^2)$,
$(x^2,x^2y)$ and eight across-the-street neighbors $(yz,z^2)$,
$(xz,z^2)$, $(xz,yz)$, $(y^2,yz)$, $(x^2,xz)$, $(xy^2,xyz)$, $(x^2y,xyz)$,
and $(x^2y,xy^2)$. This yields the border web

\medskip
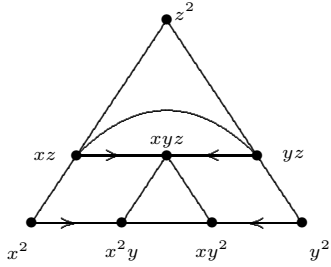
\makebox[11 true cm]{
\beginpicture
\setcoordinatesystem units <0.6cm,0.6cm>
\setplotarea x from -0.5 to 6.5, y from -0.9 to 5

\put {$\bullet$} at 0 0
\put {$\bullet$} at 2 0
\put {$\bullet$} at 4 0
\put {$\bullet$} at 6 0
\put {$\bullet$} at 1 1.5
\put {$\bullet$} at 3 1.5
\put {$\bullet$} at 5 1.5
\put {$\bullet$} at 3 4.5
\put {$\scriptstyle x^2$}  at -0.3 -0.6
\put {$\scriptstyle x^2y$} at 2 -0.6
\put {$\scriptstyle xy^2$}  at 4 -0.6
\put {$\scriptstyle y^2$}  at 6.4 -0.6
\put {$\scriptstyle xz$}  at 0.3 1.5
\put {$\scriptstyle yz$}  at 5.8 1.5
\put {$\scriptstyle xyz$}  at 3 1.8
\put {$\scriptstyle z^2$}  at 3.4 4.7
\arrow <2mm> [.2,.67] from  1.8  1.5 to 2 1.5
\arrow <2mm> [.2,.67] from  4  1.5 to 3.8 1.5
\arrow <2mm> [.2,.67] from  0.8  0 to 1 0
\arrow <2mm> [.2,.67] from  5  0 to 4.8 0
\setlinear
\putrule from 0 0 to 6 0
\putrule from 1 1.5 to 4.9 1.5
\plot
0 0 %
3 4.5 %
6 0 %
/
\plot
2 0 %
3 1.5 %
4 0 %
/
\setquadratic
\plot
1 1.5 %
3 2.5 %
5 1.5 %
/

\endpicture}
\medskip

\noindent where we have marked next-door neighbors by arrows.
Since we have $x b_2=y b_3 = b_6$, we can use part~a) of the
proposition and remove one of the sets $\AS(2,3)$, $\ND(2,6)$,
or $\ND(3,6)$ from the system of generators of~$I(\mathbb{B}_{\mathcal{O}})$.
Although there are many further ``loops'' in the remaining
part of the border web, we may use \cocoa\ to check that
no other set $\ND(i,j)$ or $\AS(i,j)$ can be
removed without changing the generated ideal.
\end{example}

Using the second part of the proposition, we can remove some generators
of~$I(\mathbb{B}_{\mathcal{O}})$ in Example~\ref{exa1xyz}.

\begin{example}\label{exa1xyzcont}
Consider $P=K[x,y,z]$ and $\mathcal{O}=\{1,x,y,z\}$ with the
border web explained in Example~\ref{exa1xyz}. Then the
border terms $b_2=xy$, $b_3=xz$ and $b_5=yz$ satisfy
$zb_2 = y b_3= x b_5$. Therefore one of the sets
$\AS(2,3)$, $\AS(2,5)$, or $\AS(3,5)$ can be removed from
the system of generators of~$I(\mathbb{B}_{\mathcal{O}})$
without changing the ideal. As already explained in
Example~\ref{exa1xyz}, none of the remaining sets
$\AS(i,j)$ can be removed thereafter.
\end{example}

\bigbreak
%
%

\section{The Homogeneous Border Basis Scheme}
\label{The Homogeneous Border Basis Scheme}

Let $P=K[x_1,\dots,x_n]$ be graded by $W=(w_1\;\cdots\;w_n)
\in\Mat_{1,n}(\mathbb{N}_+)$,
let $\mathcal{O}=\{t_1,\dots,t_\mu\}$ be an order ideal,
and let $\partial\mathcal{O}=\{b_1,\dots,b_\nu\}$ be its border.
If we restrict our attention to zero-dimensional ideals
$I\subset P$ which have an $\mathcal{O}$-border basis and are homogeneous
with respect to the grading given by~$W\!$, we obtain the following
subscheme of the border basis scheme.

\begin{definition}\label{defHBB}
Let $\{c_{ij} \mid 1\le i\le \mu,\;
1\le j\le\nu\}$ be a set of further indeterminates.

\begin{items}
\item The {\em generic homogeneous $\mathcal{O}$-border prebasis}
is defined to be the set of polynomials
$G=\{g_1,\dots,g_\nu\}$ in the ring
$K[x_1,\dots,x_n,c_{11},\dots,c_{\mu\nu}]$ where
$$
g_j = b_j -\sum_{\{i\in\{1,\dots,\mu\}\mid\deg_W(t_i)=\deg_W(b_j)\}} c_{ij}t_i
$$
for $j=1,\dots,\nu$.

\item For $k=1,\dots,n$, let $\mathcal{A}_k \in\Mat_{\mu}(K[c_{ij}])$ be the
$k^{\rm th}$ formal multiplication matrix associated to~$G$.
It is also called the $k^{\rm th}$ {\em generic homogeneous multiplication matrix}
with respect to~$\mathcal{O}$.

\item The affine scheme $\mathbb{B}_{\mathcal{O}}\hom \subseteq \mathbb{A}^{\mu\nu}$
defined by the ideal $I(\mathbb{B}_{\mathcal{O}}\hom)$ generated by the entries of
the matrices $\mathcal{A}_k \mathcal{A}_\ell -\mathcal{A}_\ell \mathcal{A}_k$ with
$1\le k<\ell\le n$ is called the {\em homogeneous $\mathcal{O}$-border basis scheme}.

\end{items}
\end{definition}

Clearly, the homogeneous border basis scheme is the intersection
of~$\mathbb{B}_{\mathcal{O}}$ with the linear space
$\mathcal{Z}(c_{ij}\mid \deg_W(t_i)\ne\deg_W(b_j))$.

\begin{remark}
Let us equip $K[x_1,\dots,x_n,c_{11},\dots,c_{\mu\nu}]$
with the grading defined by the matrix~$\overline{W}$ for which
$\deg_{\overline{W}}(c_{ij})=0$ and $\deg_{\overline{W}}(x_i)=w_i$.

\begin{items}
\item The matrix~$\mathcal{A}_k$ is a homogeneous matrix in the sense
of~\cite{KR2}, Def.~4.7.1, with respect to the degree pair
given by $(\deg_W(t_1),\dots,\deg_W(t_\mu))$ for the rows and
$(\deg_W(x_k t_1),\dots,\deg_W(x_k t_\mu))$ for the columns.

\item As explained in~\cite{KR2}, p.\ 118, we can add a vector
$d\in\mathbb{Z}^\mu$ to a degree pair and still have a degree pair
for the same homogeneous matrix. Thus the matrix $\mathcal{A}_\ell$
also has the degree pair given by $(\deg_W(x_k t_1),\dots,\deg_W(x_k t_\mu))$
for the rows and $(\deg_W(x_k x_\ell t_1),\dots,\deg_W(x_k x_\ell t_\mu))$
for the columns. In this way we see that both $\mathcal{A}_k\mathcal{A}_\ell$
and $\mathcal{A}_\ell \mathcal{A}_k$ are homogeneous matrices with respect to
the degree pair given by $(\deg_W(t_1),\dots,\deg_W(t_\mu))$ for the rows and
$(\deg_W(x_k e_\ell t_1),\dots,\deg_W(x_k x_\ell t_\mu))$ for the columns.
Consequently, also the commutator $\mathcal{A}_k\mathcal{A}_\ell
-\mathcal{A}_\ell\mathcal{A}_k$ is a homogeneous matrix with respect to
this degree pair.

\end{items}
\end{remark}

In order to deform a homogeneous ideal having an $\mathcal{O}$-border basis
to its border form ideal, we may try to construct a suitable
rational curve inside the homogeneous border basis scheme.
If~$\mathcal{O}$ has a $\maxdeg_W$ border (see Definition~\ref{DefMaxdeg}),
this plan can be carried out as follows.

\begin{theorem}{\bf (Homogeneous Maxdeg Border Bases)}\label{homcommute}\\
Suppose that the order ideal~$\mathcal{O}$ has a $\maxdeg_W$ border.

\begin{items}
\item The generic homogeneous multiplication matrices commute.

\item Let $d=\max\{\deg_W(t_1),\dots,\deg_W(t_\mu)\}$, let
$r=\# \{t\in\mathcal{O}\mid \deg_W(t)=d\}$, and let
$s=\# \{t\in \partial\mathcal{O}\mid \deg_W(t)=d\}$.
Then the homogeneous border basis scheme $\mathbb{B}_{\mathcal{O}}\hom$
is an affine space of dimension $r\,s$.

\item If $I\subset P$ is a homogeneous ideal which has an $\mathcal{O}$-border
basis $G=\{g_1,\dots,g_\nu\}$, then there exists a flat family
$K[z]\To K[z][x_1,\dots,x_n]/J$
such that $\mathcal{O}$ is a $K[z]$-basis of the right-hand side,
such that $J\vert_{z \mapsto 1}\cong I$, and such that
$J\vert_{z\mapsto 0}\cong ( b_1,\dots,b_{\nu})$.
In fact, the ideal~$J$ may be defined by
writing $g_j=b_j-\sum_{i=1}^\mu c_{ij}t_i$ and replacing $c_{ij}\in K$ by
$c_{ij}\,z\in K[z]$ for all $i,j$.

\end{items}
\end{theorem}

\begin{proof}
To prove claim~a), we examine the entry
at position $(\alpha,\beta)$ of a product $\mathcal{A}_k\mathcal{A}_\ell$.
Let $\mathcal{A}_k=(a_{ij})$ and $\mathcal{A}_\ell=(a'_{ij})$.
We want to examine the element
$\sum_{\gamma=1}^\mu a_{\alpha\gamma}a'_{\gamma\beta}$.
If $a'_{\gamma\beta}\ne 0$, the term $t_\gamma$ is contained in the support of
the representation of $x_\ell\,t_\beta$ in terms of the basis~$\mathcal{O}$.
Since~$(g_1,\dots,g_\nu)$ is a homogeneous ideal with respect to the grading on
$K[x_1,\dots,x_n,c_{11},\dots,c_{\mu\nu}]$ defined by the matrix~$\overline{W}$
for which $\deg_{\overline{W}}(c_{ij})=0$ and
$\deg_{\overline{W}}(x_i)=\deg_W(x_i)=w_i$, we have the relations
$\deg_W(t_\gamma)=\deg_W(x_\ell t_\beta) > \deg_W(t_\beta)$.
For the same reason, if $a_{\alpha \gamma}\ne 0$ we have the relations
$\deg_W(t_\alpha)=\deg_W(x_k  t_\gamma) > \deg_W(t_\gamma)$.
We deduce the inequality $\deg_W(t_\alpha) > \deg_W(x_\ell t_\beta)$.
Hence the assumption that~$\mathcal{O}$ has a $\maxdeg_W$ border
implies $x_\ell t_\beta \notin \partial\mathcal{O}$.
We conclude that $x_\ell t_\beta \in \mathcal{O}$, $t_\gamma = x_\ell t_\beta$,
and hence $a'_{\gamma \beta} = 1$.
Therefore, in order to get $a_{\alpha \gamma} a'_{\gamma \beta} \ne 0$
in the sum above, we need to have $a'_{\gamma \beta} = 1$
and $t_\gamma=x_\ell t_\beta$. In particular, this condition fixes~$\gamma$.

If the surviving summand $a_{\alpha \gamma}$ of
$\sum_{\gamma=1}^\mu a_{\alpha\gamma}a'_{\gamma\beta}$
is not zero, there are two possibilities.
Either we have $t_\alpha = x_k t_\gamma$ and thus $a_{\alpha \gamma} = 1$,
or we have $x_k t_\gamma = b_j$, $t_\alpha \in {\rm Supp}(b_j - g_j)$,
and hence $a_{\alpha \gamma} = c_{\alpha j}$.
In the first case, we have $t_\alpha = x_k x_\ell t_\beta$.
In the second case, we have $x_k x_\ell t_\beta = b_j$ and
$t_\alpha \in {\rm Supp}(b_j - g_j)$.
Now it is clear that if we examine the product
$\mathcal{A}_\ell\mathcal{A}_k$, we get the same conditions.
Therefore we conclude that
$\mathcal{A}_k\mathcal{A}_\ell = \mathcal{A}_\ell\mathcal{A}_k$.

Next we show~b). The entries of the
commutators $\mathcal{A}_k\mathcal{A}_\ell -\mathcal{A}_\ell\mathcal{A}_k$
are the defining equations of the scheme~$\mathbb{B}_{\mathcal{O}}\hom$ in the
affine subspace $\mathcal{Z}(c_{ij}\mid \deg_W(t_i)\ne\deg_W(b_j))$
of~$\mathbb{A}^{\mu\nu}$.
By~a), these commutators are all zero. The number $r\,s$ is precisely
the dimension of this affine subspace.

To show~c), it now suffices to connect the given point
in this affine space by a line to the origin and to apply
Corollary~\ref{ratcurve}.
\end{proof}

If an ideal~$I$ has an $\mathcal{O}$-border basis and
$\mathcal{O}$ has a $\maxdeg_W$ border for some grading
given by a matrix $W\in\Mat_{1,n}(\mathbb{N}_+)$, we can
combine the two flat families of Theorem~\ref{DFdeform}
and part~c) of the theorem above. As an illustration, we continue
the discussion of Example~\ref{deftoDFex}.

\begin{example}\label{exdefcontinued}
Let~$I$ be the ideal $I=( x^2+xy -\frac{1}{2}y^2-x-\frac{1}{2}y,\, y^3-y,\,
xy^2-xy)$ in~$K[x,y]$, where ${\rm char}(K)\ne 2$, and
let $\mathcal{O}=\{1,x,x^2,y,y^2\} \subset \mathbb T^2$.
Using the fact that~$\mathcal{O}$ has a
$\maxdeg_W$ border with respect to the standard grading, we have
already deformed~$I$ to~$\DF_W(I)=( x^3,\, x^2y,\,
xy+x^2-\frac{1}{2}y^2,\, xy^2,\, y^3)$.

Now we apply the theorem. We equip the summands $x^2$ and $y^2$ in the
third polynomial with a factor~$z$ and get $J=( x^3,\, x^2y,\,
xy+zx^2-\frac{1}{2}zy^2,\, xy^2,y^3)$.
As we now let $z\To 0$, we get the border form ideal of~$I$. This is a
flat deformation by part~c) of the theorem. We can also directly check that
the multiplication matrices
$$
\mathcal{A}_x= \begin{pmatrix}
0 & 0 & 0 & 0 & 0 \cr
1 & 0 & 0 & 0 & 0 \cr
0 & 1 & 0 & -z & 0 \cr
0 & 0 & 0 & 0 & 0 \cr
0 & 0 & 0 & \frac{1}{2}z & 0
\end{pmatrix}
\quad \hbox{\rm and}\quad \mathcal{A}_y =
\begin{pmatrix}
0 & 0 & 0 & 0 & 0 \cr
0 & 0 & 0 & 0 & 0 \cr
0 & -z & 0 & 0 & 0 \cr
1 & 0 & 0 & 0 & 0 \cr
0 & \frac{1}{2}z & 0 & 1 & 0
\end{pmatrix}
$$
commute as elements of $\Mat_5(K[z])$.
\end{example}

Notice that, at least following the approach taken here, it is not possible
to connect~$I$ to~$\BT_{\mathcal{O}}$ using just one irreducible rational
curve on the border basis scheme.
The next example shows that the $\maxdeg$ border property
is indispensable for the theorem to hold.

\begin{example}\label{secondexample}
The order ideal
$\mathcal{O}=\{1,x,y,x^2,xy,y^2,x^2y,xy^2,x^2y^2\}\subseteq \mathbb{T}^2$.
does not have a $\maxdeg_W$ border
with respect to any grading given by a matrix $W\in\Mat_{1,2}(\mathbb{N}_+)$.
The generic homogeneous $\mathcal{O}$-border basis is
$G=\{g_1,\dots,g_6\}$ with
$g_1 = y^3-c_{71}x^2y - c_{81}xy^2$, $g_2=x^3-c_{72}x^2y -c_{82}xy^2$,
$g_3=xy^3-c_{93}x^2y^2$, $g_4=x^3y-c_{94}x^2y^2$, $g_5=x^2y^3$, and
$g_6=x^3y^2$.

For the defining ideal of~$\mathbb{B}_{\mathcal{O}}\hom$, we find
$( c_{82}c_{93}+c_{72}-c_{94},\, c_{71}c_{94}+c_{81}-c_{93})$.
Hence~$\mathbb{B}_{\mathcal{O}}\hom$ is not a 2-dimensional
affine space (as would be the case if the theorem were applicable),
but isomorphic to a 4-dimensional affine space
via the projection to $\mathcal{Z}(c_{72},c_{81})$.
\end{example}

Another consequence of the theorem is that the homogeneous border basis
scheme can have a dimension which is higher than $n\mu$, the natural
generalization of the dimension of~$\mathbb{B}_{\mathcal{O}}$ for $n=2$
(see Remark~\ref{BBSprops}).

\begin{example}{\bf (Iarrobino)}\label{exIarrobino}
In the paper~\cite{I} Iarrobino proves that Hilbert schemes need not be irreducible
(see also~\cite{MS}, Theorem 18.32).
In particular, he produces an example which can easily be explained
using homogeneous border basis schemes. Let $\mathcal{O}$
be an order ideal in~$\mathbb{T}^3$ consisting of all terms of
degree $\le 6$ and 18 terms of degree seven.
The we have $d=7$ and $r=s=18$ in part~b) of the theorem.
Hence~$\mathbb{B}_{\mathcal{O}}\hom$ is isomorphic to an affine space
of dimension 324. In particular, it follows that $\dim(\mathbb{B}_{\mathcal{O}})
\ge 324$. On the other hand, the irreducible component
of~$\mathbb{B}_{\mathcal{O}}$ containing the points corresponding to
reduced ideals has dimension $3\cdot\mu=3\cdot 102=306$.
\end{example}

In the maxdeg border case, we can also compare the dimension
of~$\mathbb{B}_{\mathcal{O}}\hom$
to the dimension of the {\em zero fiber}~$Z$, i.e.\ the dimension of
the subscheme of~$\mathbb{B}_{\mathcal{O}}$ parametrizing schemes
supported at the origin.
Since~$\mathbb{B}_{\mathcal{O}}\hom$ is contained in~$Z$,
the preceding example implies that the dimension of~$Z$ can be
larger than~$n\mu$, the dimension of the irreducible component
of~$\mathbb{B}_{\mathcal{O}}$ containing the points corresponding to
reduced ideals. For $n=2$, a more precise estimate is available.

\begin{example}
Let $n=2$. Then the dimension of~$Z$ is~$\mu-1$ by~\cite{B}.
If~$\mathcal{O}$ has a maxdeg border then the theorem yields $s=d+1-r$
and $\dim(\mathbb{B}_{\mathcal{O}}\hom)=r(d+1-r)\le (\frac{d+1}{2})^2$.
This agrees with $\mathbb{B}_{\mathcal{O}}\hom\subseteq Z$ since
$(\frac{d+1}{2})^2 \le \frac{d(d+1)}{2}+r-1 = \mu-1$.
\end{example}

Let us end this section with an example application of Theorem~\ref{homcommute}.

\begin{example}
In~\cite{MS}, Example 18.9, the authors consider the
ideal $I=(x^2-xy,\,\allowbreak y^2-xy,\, x^2y,\, xy^2)$
in the ring $\mathbb{C}[x,y]$.
It has a border basis with respect to the order ideal
$\mathcal{O} = \{1,x,y,xy\}$, i.e.\ it corresponds to a point
in~$\mathbb{B}_{\mathcal{O}}$. It is clear that no matter
which term ordering~$\sigma$ one chooses,
it is not possible to get $\mathcal{O}_\sigma(I) = \mathcal{O}$, since
$x^2>_\sigma xy$ implies $xy >_\sigma y^2$, and therefore
$xy \notin \mathcal{O}_\sigma(I)$.
The consequence is that if one wants to connect~$I$ to
a monomial ideal in the Hilbert scheme, the deformation to~$\LT_\sigma(I)$
with respect to any term ordering~$\sigma$
leads to a monomial ideal which is not $(x^2, y^2)$,
i.e.\ not in~$\mathbb{B}_{\mathcal{O}}$.

On the other hand, by Example~\ref{affinecell}, we know
that it is possible to deform the ideal~$I$ to $(x^2, y^2)$.
But we can do even better: since the ideal~$I$ is homogeneous,
it belongs to the family parametrized by the homogeneous border basis scheme
$\mathbb{B}_{\mathcal{O}}\hom$ which is an
affine space by Theorem~\ref{homcommute}.
The full family of homogeneous ideals is $(x^2-zaxy,\ y^2-zbxy,\ x^2y,\ xy^2)$.
Putting $a = b = 1$, we get the desired flat deformation
$\Phi: K[z] \To \mathbb{C}[x,y,z]/(x^2-zxy,\, y^2-zxy,\, x^2y,\, xy^2)$.
\end{example}

\bigbreak

\subsection*{Acknowledgements}
Part of this work was conducted during the Special Semester on Gr\"obner
Bases, February 1 to July 31, 2006, organized by the RICAM Institute
(Austrian Academy of Sciences) and the RISC Institute (Johannes Kepler
University) in Linz, Austria. The authors thank these institutions, and in
particular Prof.\ Bruno Buchberger, for the warm hospitality and the
excellent work atmosphere they experienced in Linz.

\bigbreak

\end{document}